\documentclass{sl3}     %% This the CLASS for Studia Logica,
\usepackage{slsec}     %% PACKAGES for Studia Logica
\usepackage{stud_log} 
\usepackage{slfoot2}
\usepackage{slthm}     %% Theorem-like definitions as in amsthm,
\usepackage{amsmath}
\usepackage{amssymb,amsfonts}
\usepackage{csquotes}
\usepackage{setspace}

\usepackage{titling}
\usepackage{url,color,amsfonts}
\usepackage{authblk,abstract}
\usepackage{comment}
\usepackage{tikz}
\usepackage{tikz-cd}
\usetikzlibrary{shapes}
\usetikzlibrary{plotmarks}
\usepackage{mathtools}
\usepackage{stmaryrd} %For \bigcurlywedge

\usepackage{graphicx}%
\usepackage{multirow}%
\usepackage{amsmath,amssymb,amsfonts}%
\usepackage{xcolor}%
\usepackage{scalerel}

\usepackage{enumerate}
\newcommand{\m}[1]{{\bf {#1}}}									% Algebras
\newcommand{\cls}[1]{\ensuremath{{\sf #1}}}						% Varieties, Calligraphic
\newcommand{\pc}[1]{\ensuremath{{\sf #1}}}						% Proof Systems, Sans Serif	
\newcommand{\opr}[1]{\ensuremath{{\mathbb #1}}}					% Class operator	
\newcommand{\prp}[1]{{\sf #1}}

\newcommand{\tuple}[1]{\ensuremath{\langle{#1}\rangle}}

\newcommand{\N}{\mathbb{N}}
\newcommand{\Z}{\mathbb{Z}}
\newcommand{\R}{\mathbb{R}}

\newcommand*{\Nsum}{\boxplus}
\newcommand*{\nsum}{\boxplus}

\newcommand{\Com}[2]{\m{C}^{#2}_{#1}} %summand algebra
\newcommand{\com}[2]{C^{#2}_{#1}} %universe of summand algebra

\newcommand{\Con}[1]{\mathrm{Con}(#1)}

\newcommand{\hm}{\opr{H}}
\newcommand{\iso}{\opr{I}}
\newcommand{\pr}{\opr{P}}

\newcommand{\pu}{\opr{P}_\mathsf{U}}
\newcommand{\sub}{\opr{S}}
\newcommand{\vr}{\opr{V}}
\newcommand{\hspu}{\hm\sub\pu}

\newcommand{\V}{\cls{V}}
\newcommand{\K}{\cls{K}}

\newcommand{\Vfc}{\V_{\mathrm{fc}}}
\newcommand{\Vfsi}{{\cls{V}_{\mathrm{FSI}}}}
\newcommand{\Vsi}{{\cls{V}_{\mathrm{SI}}}}
\newcommand{\Kfsi}{{\cls{K}_{\mathrm{FSI}}}}
\newcommand{\Ksi}{{\cls{K}_{\mathrm{SI}}}}

\newcommand{\chain}[1]{{#1}_{\mathrm{C}}}

\newcommand{\under}{\backslash}
\newcommand{\ovr}{/}

\newcommand{\ut}{\textrm{\textup{e}}}
\newcommand{\zr}{\textrm{\textup{f}}}
\newcommand{\meet}{\mathbin{\land}}
\newcommand{\join}{\mathbin{\lor}}

\newcommand{\alg}[1]{\langle #1 \rangle}
\newcommand{\pair}[1]{\langle #1 \rangle}

\newcommand{\The}{\Theta}

\newcommand{\SemRL}{\pc{SRL}}					%Semilinear RL
\newcommand{\CSemRL}{\pc{CSRL}}				%Commutative SRL
\newcommand{\nSRL}[1]{{#1}\pc{SRL}}		%n-potent SRL (n is input)
\newcommand{\nCSRL}[1]{{#1}\pc{CSRL}}			%n-potent SCRL (n is input)

\newcommand{\BL}{\pc{BL}}							%BL-algebras
\newcommand{\BH}{\pc{BH}}							%Basic hoops
\newcommand{\MV}{\pc{MV}}							%MV-algebras
\newcommand{\WH}{\pc{WH}}					%Wajsberg hoops
\newcommand{\MTL}{\pc{MTL}}						%MTL-algebras

\newcommand{\CanSRL}{\pc{CanSRL}}			%Semilinear CanRL
\newcommand{\CCanSRL}{\pc{CanCSRL}}		%Semilinear CanCRL

\newcommand{\LG}{\pc{LG}}								%l-groups
\newcommand{\ALG}{\pc{ALG}}						%Abelian l-groups
\newcommand{\RLG}{\pc{RLG}}						%Representable l-groups

\newcommand{\GA}{\pc{GA}}  					 	%Gödel algebras
\newcommand{\RSA}{\pc{RSA}}						%Relative Stone algebras
\newcommand{\SM}{\pc{SM}}							%Sugihara monoids
\newcommand{\OSM}{\pc{OSM}}					%Odd sugihara monoids

\newcommand{\slat}{\Lambda}
\newcommand{\amal}{\Omega}

\newcommand{\starm}{{\scaleobj{0.75}{\wedge}}}

\newcommand{\Ln}[1]{\pmb{\textup{\L}}_{#1}}
\newcommand{\Lnw}[1]{\pmb{\textup{\L}}_{{#1},\omega}}
\newcommand{\Wn}[1]{\m{W}_{#1}}
\newcommand{\Wnw}[1]{\m{W}_{{#1},\omega}}

\DeclareMathOperator{\luk}{\mathrm{\textup{\L}}}
\DeclareMathOperator{\wajs}{Wajs}
\DeclareMathOperator{\basic}{Basic}
\DeclareMathOperator{\basicfc}{Basic_{\mathrm{fc}}}

\renewcommand{\theequation}{\Roman{equation}}
\newtheorem{theorem}{Theorem}[section]
\newtheorem{lemma}[theorem]{Lemma}
\newtheorem{proposition}[theorem]{Proposition}
\newtheorem{corollary}[theorem]{Corollary}

\theoremstyle{definition}

\newtheorem{remark}[theorem]{Remark}
\newtheorem{question}[theorem]{Question}
\newtheorem{example}[theorem]{Example}

\HeadingsInfo{W. Fussner and S. Santschi}{Amalgamation in Semilinear Residuated Lattices}

\begin{document}

\setcounter{page}{1}     %%  Initial page number. Do not change. 

%%%%%%  FORMAT OF AUTHOR AND TITLE INFORMATION:
 
%%%%%%%%%%%%%%%%%%% The (default) case of one author
%%%%%%%%%%%%%%%%%%% and one or two lines of title
%\AuthorTitle{Wesley Fussner}{Poset products as relational models}
%%%%%%%%%%%%%%%%%%% The case of two authors
%%%%%%%%%%%%%%%%%%% and one line title
\twoAuthorsTitleoneline{Wesley Fussner}{Simon Santschi}{Amalgamation in Semilinear Residuated Lattices}

%%%%%%%%%%%%%%%%%%% The case of two authors
%%%%%%%%%%%%%%%%%%% and at least two lines title
%\twoAuthorsTitle{Wesley Fussner}{ sdfas}{The First 
%Line of The Title\linebreak 
%and the Second Line}

%%%%%%%%%%%%%%%%%%% The case of 
%%%%%%%%%%%%%%%%%%% three authors & one or two lines title 
%\threeAuthorsTitle{M.\ts Smith}{W.\ts A. Novak}{C. 
%\ts Johns}{The first 
%line of the title
%\linebreak and the second line}

%%%%%%%%%%%%%%%%%%% The case of 
%%%%%%%%%%%%%%%%%%% three authors & at least three lines title 
%\threeAuthorsTitlethreelines{M.\ts Smith}{W.\ts A. Novak}
%{C. \ts Johns}{The First Line of The Title\linebreak 
%The Second Line of The Title,\linebreak
%and The Third Line}

%%%%%%  INFORMATION FOR FOOTER OF FIRST PAGE
%%%%%%  will be inserted by the Editorial Office.

\renewcommand{\footnoterule}{\noindent\rule{5cm}{0.4pt}{\vspace{5pt}}}%
\renewcommand\thefootnote{\arabic{footnote}}%

%\begin{center}
%{\Large\bf Amalgamation in Semilinear Residuated Lattices}\\
%\vspace{6pt}
%{Wesley Fussner and Simon Santschi}\\
%\vspace{12pt}
%{\footnotesize Affiliations}
%\vspace{6pt}
%\end{center}
   
%\PresentedReceived{}{}

\setcounter{footnote}{0}%

%%%%%%  ABSTRACT AND KEYWORDS  (obligatory)

%\begin{abstract}
{\setstretch{0.85}
\begin{displayquote}
{\footnotesize {\bf Abstract.} We survey the state of the art on amalgamation in varieties of semilinear residuated lattices. Our discussion emphasizes two prominent cases from which much insight into the general picture may be gleaned: idempotent varieties and their generalizations ($n$-potent varieties, knotted varieties), and cancellative varieties and their relatives (MV-algebras, BL-algebras). Along the way, we illustrate how general-purpose tools developed to study amalgamation can be brought to bear in these contexts and solve some of the remaining open questions concerning amalgamation in semilinear varieties. Among other things, we show that the variety of commutative semilinear residuated lattices does not have the amalgamation property. Taken as a whole, we see that amalgamation is well understood in most interesting varieties of semilinear residuated lattices, with the last few outstanding open questions remaining principally in the cancellative setting.}
\end{displayquote}}
%\end{abstract}


%%%%%  THE BODY OF THE PAPER

\section{Introduction}\label{sec:intro}

The amalgamation property is of fundamental importance throughout algebra and model theory, and was first considered in the context of groups by Schreier in 1927 \cite{Schreier1927}. In the last one hundred years, amalgamation has been studied in an enormous range of environments; the reader may consult \cite{Kiss1983} for a general survey of what was known as of the early 1980s.

In the particular context of residuated lattices, the study of amalgamation has been primarily inspired by two strands of research. First, the latter half of the twentieth century witnessed the gradual development of systematic links between amalgamation and various kinds of syntactic interpolation properties, themselves of widespread interest in logic; see \cite{Metcalfe2014,KiharaOno2010,Mak77} for a sample of this literature. Amalgamation in varieties of residuated lattices thus has considerable importance to the study of interpolation properties in substructural logics, for which residuated lattices provide algebraic models. Second, residuated lattices may be viewed quite naturally as generalizations of lattice-ordered groups, and amalgamation in the latter environment is the subject of a long and intriguing literature; see, e.g.,  \cite{PowellTsinakis1989b,PowellTsinakis1983,PowellTsinakis1989a,Glass1984,Pierce1972a,Pierce1972b}. Deepening our understanding of amalgamation in residuated lattices both sheds new light on this theory, as well as extends it in novel directions.

Although significant progress has been made in the last several decades, our understanding of amalgamation in classes of residuated lattices remains in its infancy, even for extremely well-behaved classes such as varieties. However, in the case of semilinear residuated lattices---i.e., those which are subdirect products of totally ordered residuated lattices---we have gained a tremendous amount of understanding, particularly in the last decade. We now have a very detailed understanding of amalgamation in most of the interesting varieties of semilinear residuated lattices, with the few outstanding open problems confined to the vicinity of cancellative semilinear residuated lattices. In this paper, we survey what is known about amalgamation in varieties of semilinear residuated lattices, both providing a synthesis of the progress in recent years and the tools that facilitated this progress, as well as pointing to the remaining open questions.

The paper proceeds as follows. First, in Section~\ref{sec:background}, we will lay out background information regarding universal algebra and residuated lattices that are used later on. Although we expect that the readers of this paper will be familiar with this material, we will take this occasion to fix notation. Then, in Section~\ref{sec:tools}, we will exhibit some general-purpose tools for the study of amalgamation that have shed light on the amalgamation property in semilinear residuated lattices in particular. The tools that we exhibit in Section~\ref{sec:tools} have been crucial to the progress made in the last decade, and we anticipate that similar techniques will play a further role as our understanding of amalgamation in residuated lattices expands out from the semilinear case.

In the remainder of the paper, we will discuss amalgamation in the context of particular varieties of semilinear residuated lattices. As this part of the paper will show, we are now able to give a rather sharp picture of how amalgamation works among most of the interesting semilinear varieties. In Section~\ref{sec:idempotent}, we discuss the classical cases of semilinear Heyting algebras (Section~\ref{sec:godel}) and Sugihara monoids (Section~\ref{sec:sugihara}), offering characterizations of varieties of these with the amalgamation property that take advantage of modern techniques to greatly simplify the proofs of well-known results. We also discuss the wider context of idempotent semilinear varieties in Section~\ref{sec:general}, emphasizing the surprising role of commutativity in this context. In Section~\ref{sec:knotted}, we discuss what can be said about amalgamation when idempotence is relaxed, replacing it by $n$-potence or, more generally, knotted inequalities. Here we see that the picture is rather different from in the idempotent case. In particular, we show that the variety of all commutative semilinear residuated lattices lacks the amalgamation property and even upon adjoining most knotted inequalities (in particular, integrality) and/or involution. The paper concludes in Section~\ref{sec:cancellative}, where we discuss what is known about amalgamation among cancellative semilinear varieties and varieties related to them. Here our understanding is much murkier than in the other cases. We sketch what is known about amalgamation in lattice-ordered groups and cancellative residuated lattices generally (Section~\ref{sec:lgroups}). Then, we go on to discuss two adjacent cases: That of MV-algebras and Wajsberg hoops (Section~\ref{sec:MV}) and BL-algebras and basic hoops (Section~\ref{sec:BL}). We see that in all of these cases we now have a complete understanding of amalgamation in arbitrary subvarieties.

\section{Universal Algebra and Residuated Lattices}\label{sec:background}

We will begin by laying out some key definitions and concepts, primarily to acquaint the reader with our notation. Our assumption is that readers are familiar with the elements of universal algebra, lattice theory, and residuated structures. Background on these topics may be found in \cite{BP1981}, \cite{DaveyPriestley2002}, and \cite{MPT23,GalatosJipsenKowalskiOno2007}, respectively.

Given any algebra $\m{A}$, we will write $\Con{\m A}$ for the lattice of congruences of $\m{A}$. The usual universal-algebraic class operators of homomorphic images, isomorphic images, subalgebras, products, and ultraproducts will be denoted by $\hm$, $\iso$, $\sub$, $\pr$, and $\pu$, respectively. If $\m{A}$ is an algebra and $(a,b)\in A^2$, we denote by $\Theta(a,b)$ the least member of $\Con{\m{A}}$ that contains $(a,b)$. For any class $\K$ of algebras, the classes of subdirectly irreducible and finitely subdirectly irreducible algebras in $\K$ will be respectively denoted by $\Ksi$ and $\Kfsi$. We write $\slat(\V)$ for the lattice of subvarieties of the variety $\V$. 

A \emph{residuated lattice} is an algebra ${\m A} = \alg{A,\meet,\join,\under,\ovr,\ut}$ of type $(2,2,2,2,0)$ such that $\alg{A,\meet,\join}$ is a lattice, $\alg{A,\cdot,\ut}$ is a monoid, and for all $x,y,z\in A$,
\[
y\leq x\under z \iff x\cdot y\leq z \iff x\leq z\ovr y.
\]
Residuated lattices make up a variety $\sf RL$. We will usually write $xy$ for $x\cdot y$. 

This paper focuses on \emph{semilinear} residuated lattices. The latter are residuated lattices that are subdirect products of totally ordered ones. Semilinearity is equivalent to satisfaction of the equation
\[
(z\under (x\ovr (x\join y))z\meet\ut)\join (w(y\ovr (x\join y))\ovr w\meet \ut)\approx \ut,
\]
so semilinear residuated lattices form a subvariety $\mathsf{\SemRL}$ in $\slat (\mathsf{RL})$; see \cite[Theorem~2.2]{GC2004} and \cite[Section~6.1]{BlountTsinakis2003}. 
Appendix~\ref{app:nomenclature} contains a summary of the naming conventions for subvarieties of $\mathsf{\SemRL}$ that we will use in this paper. This also includes a bird's-eye view of the results on amalgamation that we will discuss in this paper.

A residuated lattice is called \emph{commutative} if it satisfies $xy\approx yx$. The equation $x\under y \approx y\ovr x$ holds in any commutative residuated lattice, so it is duplicative to include both $\under$ and $\ovr$ in the signature for commutative residuated lattices. Thus, in this case we will just write $x\to y$ for the operation given by $x\under y$. The variety of semilinear commutative residuated lattices will be denoted by $\CSemRL$.

We will have occasion to consider several expansions of the basic signature of residuated lattices. First, a \emph{bounded} residuated lattice is one that has both a least and greatest element that are designated by constants in the language. Second, there are a number of cases where we may want to consider expansions of residuated lattices by an additional constant $\zr$ that is intended as a \emph{negation constant}, in the sense that we define new unary operations ${\sim}x=x\under\zr$ and $-x=\zr\ovr x$. A negation constant $\zr$ is called \emph{cyclic} if it satisfies $x\under\zr\approx\zr\ovr x$, in which case we write the common value by $\neg x$. Note that commutativity implies cyclicity. Further, a negation constant $\zr$ is called \emph{left involutive} if it satisfies $-{\sim}x\approx x$ and \emph{right involutive} if it satisfies ${\sim}-x\approx x$. A negation constant is \emph{involutive} if it is both left involutive and right involutive. When we refer to a cyclic or (left, right) involutive residuated lattice, we mean an expansion of a residuated lattice by a negation constant with the specified properties.

Residuated lattices are both congruence distributive and congruence permutable. They are also $\ut$-regular: Each congruence $\Theta$ in a residuated lattice $\m{A}$ is completely determined by the $\Theta$-class of $\ut$, which is a convex normal subuniverse of $\m{A}$. In fact, if $\m{A}$ is any residuated lattice, $\Con{\m A}$ is isomorphic to the lattice of convex normal subalgebras of $\m{A}$; see, e.g., \cite{BlountTsinakis2003}.  These facts all apply to bounded, cyclic, and (left or right) involutive residuated lattices by consider the residuated lattice reducts of these.

Note that each totally ordered residuated lattice is finitely subdirectly irreducible, since its lattice of convex normal subalgebras forms a chain. Moreover, every finite totally ordered residuated lattice is subdirectly irreducible. In fact, for any $\V\in\slat(\SemRL)$ we have that $\Vfsi$ is precisely the class of totally ordered members of $\V$. This fact is indispensable for understanding amalgamation for varieties in $\slat(\SemRL)$, and we will rely on it very often throughout our discussion.

An algebra $\m{A}$ is said to have the \emph{congruence extension property} (\prp{CEP}) if for each subalgebra $\m{B}\leq \m{A}$ and congruence $\Theta \in \Con{\m{B}}$, there exists  a congruence $\Phi \in \Con{\m{A}}$ such that $\Phi \cap B^2 = \Theta$, and we say that a class of algebras has the \prp{CEP} if each of its members has the \prp{CEP}. The \prp{CEP} will be relevant for some of the tools regarding the amalgamation property, but it is also interesting from the point of view of abstract algebraic logic, since if a variety of residuated lattices has the \prp{CEP} this is equivalent to the corresponding substructural logic having a local deduction theorem.

\begin{example}[{\cite[Example 6.1]{vanAlten2005}}]\label{ex:CEPfail}
We consider the totally ordered residuated lattice depicted in Figure~\ref{fig:CEPfail}. Here and throughout the paper we will adopt some notational conventions for reading labeled Hasse diagrams, such as the one given in Figure~\ref{fig:CEPfail}: Central elements (i.e., those that commute with all other elements) will be labeled by rounded nodes, whereas non-central elements will be labeled by square nodes. A filled node denotes an idempotent element, whereas an open node denotes an element that is not idempotent.

Referring to these conventions, one may see that the totally ordered residuated lattice $\m{A}$ depicted in Figure~\ref{fig:CEPfail} does not have the \prp{CEP}. Just note that the convex normal subalgebra $\m{B}$ of $\m{A}$ that is generated by $a$ has universe $\{\ut,a,b\}$, but the convex normal subalgebra of $\m{B}$ generated by $a$ has universe $\{\ut,a\}$.
\end{example}

Although Example~\ref{ex:CEPfail} shows that  (semilinear) residuated lattices do not have the \prp{CEP} in general, this is the case for commutative (semilinear) residuated lattices.

\begin{lemma}[{\cite[Lemma~3.57]{GalatosJipsenKowalskiOno2007}}]
Every commutative residuated lattice has the congruence extension property.
\end{lemma}

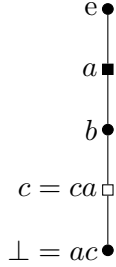
\begin{figure}
\centering
\begin{tikzpicture}[
place/.style={circle,draw=black,fill=black, minimum size = 4pt, inner sep = 0pt},
square/.style={regular polygon,regular polygon sides=4},
place2/.style={square,draw=black,fill=black, minimum size = 5.5pt, inner sep = 0pt},
place3/.style={square,draw=black, minimum size = 5.5pt, inner sep = 0pt},
place4/.style={circle,draw=black, minimum size = 4pt, inner sep = 0pt}]

   \node[place] (1) at (0,-1.6) {};
   \node[place2] (a) at (0,-2.4) {};
  \node[place] (b) at (0,-3.2) {};
  \node[place3] (c) at (0,-4) {};
  \node[place] (bot) at (0,-4.8) {}; 
   
  \node[left] () at (1) {$\ut$};
   \node[left] () at (a) {$a$};
      \node[left] () at (b) {$b$};
  \node[left] () at (c) {$c = ca$};
  \node[left] () at (bot) {$\bot = ac$};
  
  \draw (1) -- (a) -- (b) -- (c) -- (bot);
\end{tikzpicture}
\caption{Labeled Hasse diagrams of a totally ordered residuated lattice without the \prp{CEP}.}
\label{fig:CEPfail}
\end{figure}

\section{Amalgamation: Fundamentals and Tools of the Trade}\label{sec:tools}

In this section, we outline some basic methods that have proven important in understanding amalgamation in classes of semilinear residuated lattices. Although semilinear residuated lattices are our sole concern in this paper, we note that the toolkit we assemble in this section applies far beyond this setting. Our discussion is therefore framed in universal-algebraic terms, with the intention that this summary of basic techniques may be applied by the reader in other contexts. Further, the interested reader will find, upon consulting the references, that many of the results summarized in this section apply beyond the setting of varieties (e.g., to quasivarieties). However, since we are concerned with varieties in this paper, we confine our attention to that setting, where many of the results of this section have especially simple presentations.

First, some key definitions. Consider a class $\K$ of similar algebras. A \emph{span in $\K$} is an ordered pair $\langle \varphi_1 \colon \m{A} \to \m{B},\varphi_2\colon \m{A} \to \m{C} \rangle$ of injective homomorphisms $\varphi_1\colon\m{A}\to\m{B}$ and $\varphi_2\colon\m{A}\to\m{C}$, where $\m{A},\m{B},\m{C}$ are algebras in $\K$. Given a class $\mathsf{M}$ of algebras in the same signature as $\K$ and a span $\langle \varphi_1 \colon \m{A} \to \m{B},\varphi_2\colon \m{A} \to \m{C} \rangle$ in $\K$, an \emph{amalgam of $\langle\varphi_1,\varphi_2\rangle$ in $\mathsf{M}$} is an ordered pair of injective homomorphisms $\langle\psi_1\colon\m{B}\to\m{D},\psi_2\colon\m{C}\to\m{D}\rangle$ where $\m{D}\in\mathsf{M}$ and $\psi_1\circ\varphi_1=\psi_2\circ\varphi_2$, i.e., so that the diagram in Figure~\ref{fig:AP}(i) commutes. 
A \emph{one-sided amalgam} of the span $\langle \varphi_1,\varphi_2\rangle$ in $\mathsf{M}$ is an ordered pair $\langle\psi_1\colon\m{B}\to\m{D},\psi_2\colon\m{C}\to\m{D}\rangle$ where $\m{D}\in\mathsf{M}$ of an injective homomorphism $\psi_1$ and a (not necessarily injective) homomorphism $\psi_2$ such that $\psi_1\circ\varphi_1=\psi_2\circ\varphi_2$, i.e., so that the diagram in Figure~\ref{fig:AP}(ii) commutes. 

A class $\K$ of similar algebras is said to have the \emph{amalgamation property} (or \prp{AP}) if every span in $\K$ has an amalgam in $\K$, and it is said to have the \emph{one-sided amalgamation property} (or $\prp{1AP}$) if every span in $\K$ has a one-sided amalgam in $\K$. When $\V$ is a variety, we denote by $\amal({\V})$ the subposet of $\slat(\V)$ consisting of the subvarieties of $\V$ with the \prp{AP}.

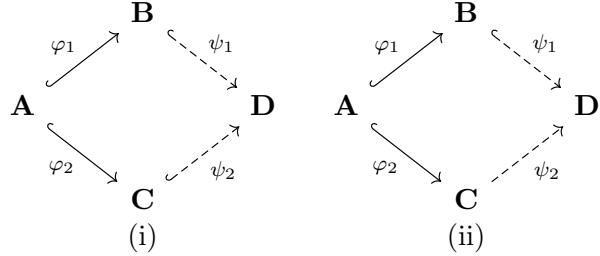
\begin{figure}
\centering
\begin{tabular}{cc}

\begin{tikzcd}
	& {\bf B} \\
	{\bf A} && {\bf D} \\
	& {\bf C}
	\arrow["{\varphi_1}", hook, from=2-1, to=1-2]
	\arrow["{\psi_1}", dashed, hook, from=1-2, to=2-3]
	\arrow["{\varphi_2}"', hook, from=2-1, to=3-2]
	\arrow["{\psi_2}"', dashed, hook, from=3-2, to=2-3]
\end{tikzcd}
&
\begin{tikzcd}
	& {\bf B} \\
	{\bf A} && {\bf D} \\
	& {\bf C}
	\arrow["{\varphi_1}", hook, from=2-1, to=1-2]
	\arrow["{\psi_1}", dashed, hook, from=1-2, to=2-3]
	\arrow["{\varphi_2}"', hook, from=2-1, to=3-2]
	\arrow["{\psi_2}"', dashed,  from=3-2, to=2-3]
\end{tikzcd} \\
(i) & (ii)
\end{tabular}
\caption{Commuting diagram illustrating the \prp{AP} and \prp{1AP}.}
\label{fig:AP}
\end{figure}

Determining whether a given variety $\V$ has the \prp{AP} is often extremely challenging. However, in \cite{Metcalfe2014} Metcalfe, Montagna, and Tsinakis give the following sufficient condition for a variety to have the \prp{AP} and use it to show the \prp{AP} for various varieties of residuated lattices.

\begin{proposition}[{\cite[Theorem 9]{Metcalfe2014}}]\label{p:condition}
Let $\K$ be a subclass of a variety $\V$ satisfying 
\begin{enumerate}[\normalfont (a)]
\item $\K$ is closed under isomorphisms and subalgebras; 
\item every subdirectly irreducible member of $\V$ belongs to $\K$;
\item for any $\m{B}\in\V$ and subalgebra $\m{A}$ of $\m{B}$, if $\The\in\Con{\m{A}}$ and  $\m{A}/\The\in \K$, then there exists a $\Phi\in\Con{\m{B}}$ such that $\Phi \cap A^2 =\The$ and $\m{B}/\Phi\in \K$; 
\item every span of in $\K$ has an amalgam in $\V$.
\end{enumerate}
Then $\V$ has the amalgamation property.
\end{proposition}

In \cite{Fussner2024}, Proposition~\ref{p:condition} was sharpened: It is shown therein that item \text{(d)} can be replaced by the demand that only \emph{finitely generated} algebras in $\K$ have an amalgam in $\V$, with the same conclusion.

Furthermore, \cite{Fussner2024} gives the following necessary and sufficient conditions for a variety to have the \prp{AP}, under the assumption that it has the \prp{CEP} and the class of its finitely subdirectly irreducible members is closed under subalgebras.

\begin{theorem}[{\cite[Corollary 3.5]{Fussner2024}}]\label{t:APmain}
Let $\V$ be any variety with the \prp{CEP} such that $\Vfsi$ is closed under subalgebras. The following are equivalent:
\begin{enumerate}[\normalfont (1)]
\item	$\V$ has the amalgamation property.
\item	$\V$ has the one-sided amalgamation property.
\item	$\Vfsi$ has the one-sided amalgamation property.
\item	Every span in $\Vfsi$ has an amalgam in $\Vfsi\times\Vfsi$.
\item Every span of finitely generated algebras in $\Vfsi$ has an amalgam in $\V$.
\end{enumerate}
\end{theorem}

Theorem~\ref{t:APmain} is very useful to check if a variety $\V$ of residuated lattices that satisfies the assumptions has the \prp{AP}, since it reduces it to checking if $\Vfsi$ has the \prp{1AP}. However, since a one-sided amalgam contains an arbitrary homomorphism, to show that  $\Vfsi$ fails the \prp{1AP} it is usually  helpful to consider spans that also force the second homomorphism in a one-sided amalgam to be injective. In \cite{FS24}, it was noticed that this can be obtained by assuming that the second embedding in the span is \emph{essential}, leading to the notion of the \emph{essential amalgamation property}.

In detail, an embedding $\varphi\colon \m{A} \to \m{B}$ is called \emph{essential} if for each homomorphism $\psi\colon \m{B} \to \m{C}$, if $\psi\circ \varphi$ is an embedding, then $\psi$ is an embedding. If $\varphi\colon \m{A} \to \m{B}$ is the inclusion map, then it is essential if and only if for each $\Theta \in \Con{\m{B}}$ with $\Theta \neq \Delta_B$, also $\Theta \cap A^2 \neq \Delta_A$. A span $\langle \varphi_1, \varphi_2 \rangle$ is called an \emph{essential span} if $\varphi_2$ is an essential embedding. We say that a class of similar algebras $\K$ has the \emph{essential amalgamation property} (or \prp{EAP}) if each essential span in $\K$ has an amalgam in $\K$.

\begin{lemma}[\cite{FS24}]\label{l:1AP-EAP}
Let $\K$ be a class of similar algebras. 
\begin{enumerate}[\normalfont (i)]
\item If $\K$ has the \prp{1AP}, then it has the \prp{EAP}.
\item If $\K$ has the \prp{EAP} and $\hm(\K) = \K$, then $\K$ has the \prp{1AP}.
\end{enumerate}
\end{lemma}
\begin{proof}
(i) Let $\langle \varphi_1 \colon \m{A} \to \m{B},\varphi_2\colon \m{A} \to \m{C} \rangle$ be an essential span in $\K$. By assumption, this span has a one-sided amalgam $\langle\psi_1\colon\m{B}\to\m{D},\psi_2\colon\m{C}\to\m{D}\rangle$  in $\K$. But then, $\psi_2\circ \varphi_2 = \psi_1 \circ \varphi_1$ is an embedding and, since $\varphi_2$ is essential, also $\psi_2$ is an embedding. So $\langle \psi_1, \psi_2 \rangle$ is an amalgam of the span.

(ii) Let $\langle \varphi_1 \colon \m{A} \to \m{B},\varphi_2\colon \m{A} \to \m{C} \rangle$ be a span in $\K$. Then, by \cite[Lemma 3]{Graetzer1971}, there exists a congruence $\Theta \in \Con{\m{C}}$ such that for the natural projection $\pi \colon \m{C} \to  \m{C}/\Theta$, the map $\pi\circ \varphi_2$ is an essential embedding. So the span $\langle \varphi_1,\pi \circ \varphi_2 \rangle$ is essential and, by assumption, it has an amalgam $\langle\psi_1\colon\m{B}\to\m{D},\psi_2\colon\m{C}/\theta \to\m{D}\rangle$ in $\K$ (see Figure~\ref{fig:EAP}). But then $\langle \psi_1, \psi_2\circ \pi \rangle$ is a one-sided amalgam of $\langle\varphi_1,\varphi_2 \rangle$.
\end{proof}

\begin{figure}
\[\begin{tikzcd}
	& {\bf B} \\
	{\bf A} && {\bf D} \\
	{\bf C}& {\bf C}/\Theta
	\arrow["{\varphi_1}", hook, from=2-1, to=1-2]
	\arrow["{\psi_1}", dashed, hook, from=1-2, to=2-3]
	\arrow["{\varphi_2}"', hook, from=2-1, to=3-1]
	\arrow[""',  hook, from=2-1, to=3-2]
	\arrow["{\pi}"',  two heads,  from=3-1, to=3-2]
	\arrow["{\psi_2}"', dashed, hook, from=3-2, to=2-3]
\end{tikzcd}
\]
\caption{Commuting diagram illustrating the proof of Lemma~\ref{l:1AP-EAP}(ii).}
\label{fig:EAP}
\end{figure}
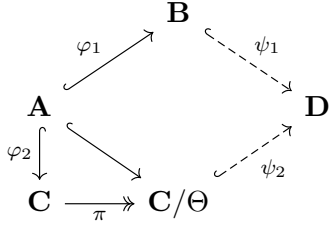

Note that for (ii) in Lemma~\ref{l:1AP-EAP}, the assumption that $\hm(\K) = \K$ is crucial for being able to factor through an essential extension. In conjunction with Theorem~\ref{t:APmain}, Lemma~\ref{l:1AP-EAP} yields the following result.

\begin{proposition}[\cite{FS24}]\label{p:ess-AP}
Let $\V$ be a variety with the \prp{CEP} and $\hm\sub(\Vfsi) = \Vfsi$. Then the following are equivalent.
\begin{enumerate}[\normalfont (1)]
\item $\V$ has the amalgamation property.
\item $\Vfsi$ has the essential amalgamation property.
\item Every essential span of finitely generated algebras in $\Vfsi$ has an amalgam in $\Vfsi$.
\end{enumerate}
\end{proposition}

For any variety of semilinear residuated lattices $\V$, we denote the class of its totally ordered members by $\chain{\V}$. It is well-known that in this case $\Vfsi = \chain{\V} = \hm\sub(\chain{\V})$. In this setting, the following corollary is immediate.

\begin{corollary}[\cite{FS24}]\label{c:ess-AP-sem}
Let $\V$ be any variety of semilinear residuated lattices with the \prp{CEP}.  The following are equivalent.
\begin{enumerate}[\normalfont (1)]
\item $\V$ has the amalgamation property.
\item $\chain{\V}$ has the essential amalgamation property.
\item Every essential span of finitely generated algebras in $\chain{\V}$ has an amalgam in $\chain{\V}$.
\end{enumerate}
\end{corollary}

As discussed in Section~\ref{sec:background}, every variety of commutative residuated lattices has the \prp{CEP}, so the previous corollary may be applied in the commutative setting. Here the result is significant enough to our discussion to state it explicitly.

\begin{corollary}[\cite{FS24}]\label{c:ess-AP-comsem}
Let $\V$ be any variety of commutative semilinear residuated lattices. The following are equivalent.
\begin{enumerate}[\normalfont (1)]
\item $\V$ has the amalgamation property.
\item $\chain{\V}$ has the essential amalgamation property.
\item Every essential span of finitely generated algebras in $\chain{\V}$ has an amalgam in $\chain{\V}$.
\end{enumerate}
\end{corollary}

In some applications, checking the \prp{EAP} directly for $\Vfsi$ of a variety $\V$ is easier if we first decompose $\Vfsi$ into subclasses and check the \prp{EAP} for these. The following brief interlude explains how this can be done.

First, we say that a subclass $\K_1$ of a class $\K_2$ of similar algebras is \emph{essentially closed in $\K_2$} if for each $\m{A} \in \K_1$ and essential embedding $\varphi\colon \m{A} \to \m{B}$ with $\m{B} \in \K_2$ also $\m{B} \in \K_1$. 

\begin{remark}
Let $\K_1$, $\K_2$, $\K_3$ be classes of similar algebras.
\begin{enumerate}[\normalfont (i)]
\item If $\K_1$ is essentially closed in $\K_2$ and $\K_3$, then $\K_1$ is essentially closed in $\K_2 \cup \K_3$.
\item If $\K_1$ and $\K_2$ are essentially closed in $\K_3$, then $\K_1 \cap \K_2$ and $\K_1 \cup \K_2$ are essentially closed in $\K_3$.
\item If $\K_1$ is essentially closed in $\K_2$ and $\K_2$ is essentially closed in $\K_3$, then $\K_1$ is essentially closed in $\K_3$.
\end{enumerate}
\end{remark}

\begin{proposition}[\cite{FS24}]
Let $\K$ be a class of algebras such that $\K = \K_1 \cup \K_2$, $\iso\sub(\K_i) = \K_i$ for $i=1,2$, and $\K_1 \cap \K_2$ is essentially closed in $\K$. 
\begin{enumerate}[\normalfont (i)]
\item If $\K_1$ and $\K_2$ have the essential amalgamation property, then $\K$ has the essential amalgamation property. 
\item If $\K$ and $\K_1 \cap \K_2$ have the essential amalgamation property, then $\K_1$ and $\K_2$ have the essential amalgamation property.
\end{enumerate} 
\end{proposition}

\begin{proof}
(i) It is enough to show that every essential span in $\K$ is an essential span in $\K_1$ or $\K_2$. So let $\langle \varphi_1 \colon \m{A} \to \m{B},\varphi_2\colon \m{A} \to \m{C} \rangle$ be an essential span in $\K$. If for $i \in \{1,2\}$,  $\m{B},\m{C} \in \K_i$, then, since $\iso\sub(\K_i) = \K_i$, also $\m{A} \in \K_i$ and $\langle \varphi_1,\varphi_2 \rangle$ is an essential span in $\K_i$. Otherwise, without loss of generality, $\m{B} \in \K_1$ and $\m{C} \in \K_2$. Thus, since $\iso\sub(\K_i) = \K_i$ for $i=1,2$, $\m{A} \in \K_1 \cap \K_2$. But then, since $\K_1 \cap \K_2$ is essentially closed in $\K$, $\m{C} \in \K_1 \cap \K_2 \subseteq \K_1$. so $\langle \varphi_1,\varphi_2 \rangle$ is an essential span in $\K_1$.

(ii) We show that $K_1$ has the \prp{EAP}. The proof for $\K_2$ is analogous. Let $\langle \varphi_1 \colon \m{A} \to \m{B},\varphi_2\colon \m{A} \to \m{C} \rangle$ be an essential span in $\K_1$. Then, by assumption this span has an amalgam $\langle \psi_1 \colon \m{B} \to \m{D}, \psi_2 \colon \m{C} \to \m{D} \rangle$ in $\K$. If $\m{D} \in \K_1$, then we are done. Otherwise, $\m{D} \in \K_2$ and, since $\iso\sub(\K_2) = \K_2$, $\m{A}, \m{B}, \m{C} \in \K_1 \cap \K_2$. Hence, by assumption, the span also has an amalgam in $\K_1 \cap \K_2 \subseteq \K_1$.
\end{proof}

Recall from Section~\ref{sec:background} that not every variety of (semilinear) residuated lattices has the \prp{CEP}. Thus, the results discussed above cannot be applied in general. The next lemma illustrates that there is still a way to transfer the failure of the \prp{AP} from the subdirectly irreducibles to the whole variety even without assuming anything about the given variety.

\begin{lemma}\label{l:doubly-essential}
Let $\V$ be a variety and let $\langle \varphi_1 \colon \m{A} \to \m{B},\varphi_2\colon \m{A} \to \m{C} \rangle$ be an essential span in $\V$ such that $\m{B}\in \Vsi$.
Then $\langle \varphi_1,\varphi_2\rangle$ has an amalgam in $\V$ if and only if it has an amalgam in $\Vsi$.
\end{lemma}

\begin{proof}
The right-to-left implication is trivial. For the converse, suppose that $\langle \varphi_1,\varphi_2\rangle$ has an amalgam $\langle \psi_1 \colon \m{B} \to \m{D},\psi_2\colon \m{C} \to \m{D} \rangle$ in $\V$. Then we may assume that $\m{D} = \prod_{i\in I} \m{D}_i$ with $\m{D}_i \in \Vsi$. Moreover, since $\m{B} \in \Vsi$, there exist $a,b \in B$ such that  $\Theta(a,b)$ is the monolith of the congruence lattice of $\m{B}$. Now, since $\psi_1$ is an embedding, there exists an $i\in I$ such that for the $i$th projection map $\pi_i \colon \m{D} \to \m{D}_i$, $(\pi_i\circ \psi_1)(a) \neq (\pi_i\circ \psi_1)(b)$. Thus $\pi_i\circ \psi_1$ is an embedding. But then also  $\pi_i\circ \psi_1 \circ \varphi_1 =  \pi_i\circ \psi_2 \circ \varphi_2$ is an embedding. Hence, since  $\varphi_2$ is essential, also $\pi_i\circ \psi_2$ is an embedding and $\langle\pi_i\circ\psi_1,\pi_i\circ\psi_2 \rangle$ is an amalgam of $\langle \varphi_1,\varphi_2\rangle$ in $\Vsi$.
\end{proof}

In particular, it follow from Lemma~\ref{l:doubly-essential} that if $\Vsi$ fails the \prp{EAP}, then $\V$ fails the \prp{AP}.

\begin{corollary}
Let $\V$ be a variety. If $\Vsi$ does not have the essential amalgamation property, then $\V$ does not have the amalgamation property.
\end{corollary}

The next method for establishing the \prp{AP} for a variety $\V$ is model-theoretic and operates by demonstrating quantifier elimination for a suitable elementary subclass. In the realm of residuated lattices, this method was notably applied in \cite{Weispfenning1989} to show that the variety of abelian $\ell$-groups has the \prp{AP} (see Section~\ref{sec:lgroups}) and in \cite{MarMet2012} to classify the varieties of Sugihara monoids with the \prp{AP} (see Section~\ref{sec:sugihara}). See also \cite{CMM11,Marchioni2012} for further applications of this technique.

Let  $\K$ be an elementary class of algebras. We say that $\K$ has quantifier elimination if for each first-order formula $\phi(\overline{x})$ there exists a quantifier-free first-order formula $\psi(\overline{x})$ such that $\K \models (\forall \overline{x})(\phi(\overline{x}) \Leftrightarrow \psi(\overline{x}))$

\begin{theorem}[{\cite[Theorem 8.4.1]{Hodges1993}}]\label{t:QE-AP}
Let $\K$ be an elementary class of algebras with quantifier elimination. Then $\iso\sub(\K)$ has the amalgamation property.
\end{theorem}

In the case of varieties of semilinear residuated lattices with the \prp{CEP}, Theorem~\ref{t:QE-AP} is usually used to show that $\chain{\V}$ has the \prp{AP} by showing quantifier elimination for a suitable elementary class $\K \subseteq \chain{\V}$ with $\iso\sub(\K) = \chain{\V}$. Then, by Corollary~\ref{c:ess-AP-sem}, it follows that $\V$ has the \prp{AP}.

Finally, we discuss some generalities about residually small and finitely generated congruence-distributive varieties, e.g., finitely generated varieties of residuated lattices.  

A variety $\V$ is called \emph{residually small} if it has up to isomorphism only a set (as opposed to a proper class) of subdirectly irreducible members. This is, in particular, the case for finitely generated congruence-distributive varieties. 

\begin{theorem}[\cite{Kearnes1989}]\label{t:AP-CEP}
Let $\V$ be a residually small congruence-modular variety with the amalgamation property. Then $\V$ has the congruence extension property.
\end{theorem}

The above theorem implies, in particular, that any finitely generated congruence-distributive variety without the \prp{CEP} does not have the \prp{AP}. In conjunction with Theorem~\ref{t:APmain}, this fact was used in \cite[Section~5]{Fussner2024} to give an effective algorithm to decide whether a finitely generated, congruence-distributive variety whose class of finitely subdirectly irreducible members is closed under subalgebras has the \prp{AP}. The following result is an immediate consequence of this.

\begin{proposition}[{cf. \cite[Theorem~5.1]{Fussner2024}}]
Let $\V\in\slat(\SemRL)$ be finitely generated. Then it is decidable whether $\V$ has the $\prp{AP}$.
\end{proposition}

Note that since every commutative residuated lattice has the \prp{CEP}, Theorem~\ref{t:AP-CEP} is only useful for varieties that contain a non-commutative residuated lattice.

We say that an algebra $\m{A}$ is \emph{strictly simple} if it is simple and it does not have any non-trivial subalgebras.

\begin{proposition}[{cf. \cite{JipsenRose1989}}]\label{p:strictly-simple}
Let $\V$ be a congruence-distributive variety that is generated by a finite strictly simple algebra $\m{A}$. Then $\V$ has the amalgamation property.
\end{proposition}

\begin{proof}
First, note that $\hm\sub(\m{A}) = \hm(\m{A})$ consists, up to isomorphism, of $\m{A}$ and a trivial algebra. So, by J\'{o}nsson's Lemma, $\Vfsi = \hm\sub(\m{A})$. Moreover, since $\V$ is congruence-distributive, by \cite[Corollary 2.4]{Fussner2024}, $\V$ has the \prp{CEP} if and only if $\Vfsi$ has the \prp{CEP}, which is clearly the case. Hence $\V$ has the \prp{CEP} and, by Proposition~\ref{p:ess-AP}, it has the \prp{AP} if and only if $\Vfsi$ has the essential \prp{AP}, which it trivially has.
\end{proof}

In particular, it follows from Proposition~\ref{p:strictly-simple} that any variety generated by a finite strictly simple residuated lattice has the \prp{AP}. Hence, this illustrates an easy way to obtain examples of varieties of residuated lattices with the \prp{AP}.  

\begin{figure}
\centering
\begin{tikzpicture}[
place/.style={circle,draw=black,fill=black, minimum size = 4pt, inner sep = 0pt},
square/.style={regular polygon,regular polygon sides=4},
place2/.style={square,draw=black,fill=black, minimum size = 5.5pt, inner sep = 0pt},
place3/.style={square,draw=black, minimum size = 5.5pt, inner sep = 0pt},
place4/.style={circle,draw=black, minimum size = 4pt, inner sep = 0pt}]

   \node[place2] (a) at (0,-1.6) {};
   \node[place] (1) at (0,-2.4) {};
  \node[place2] (b) at (0,-3.2) {};
  \node[place] (bot) at (0,-4) {};
   
  \node[left] () at (a) {$ab = a$};
   \node[left] () at (1) {$\ut$};
      \node[left] () at (b) {$ba = b$};
  \node[left] () at (bot) {$\bot$};
  
  \draw (a) -- (1) -- (b) -- (bot);
\end{tikzpicture}
\caption{Labeled Hasse diagrams of a strictly simple totally ordered residuated lattice.}
\label{fig:strictsimp}
\end{figure}
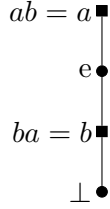

\begin{example}[\cite{GJM2020}]\label{ex:strictly-simple} 
The algebra depicted in Figure~\ref{fig:strictsimp} is a strictly simple totally ordered residuated lattice. Hence, by Proposition~\ref{p:strictly-simple}, the variety it generates has the \prp{AP}.
\end{example}

Along similar lines, we may obtain the following abstract characterization of when a variety generated by a single finite simple totally ordered residuated lattice has the \prp{AP}. 

\begin{proposition}\label{prop:simple chain}
Let $\m{A}$ be a finite simple totally ordered residuated lattice and $\V = \vr(\m{A})$. Then $\V$ has the amalgamation property if and only if  $\m{A}$ has the \prp{CEP} and does not contain two distinct isomorphic subalgebras. The same holds for bounded and/or involutive expansions.
\end{proposition}

\begin{proof}
For the left-to-right direction, first note that, by Theorem~\ref{t:AP-CEP}, $\V$ has the \prp{CEP}.  Thus, by J\'{o}nsson's Lemma,  $\Vfsi = \hm\sub(\m{A}) = \sub\hm(\m{A}) = \iso\sub(\m{A})$.
Now note that, since $\m{A}$ is simple, each span $\langle \varphi_1 \colon \m{B} \to \m{A}, \varphi_2 \colon \m{B} \to \m{A} \rangle$ is essential. By Corollary~\ref{c:ess-AP-sem}, the span $\langle \varphi_1,\varphi_2 \rangle$ has an amalgam in $\Vfsi = \iso\sub(\m{A})$. But every amalgam of this span in $\iso\sub(\m{A})$ is of the form $\langle \psi_1 \colon \m{A} \to \m{A}, \psi_2 \colon \m{A} \to \m{A} \rangle$ for cardinality reasons and, since $\m{A}$ is totally ordered, $\psi_1 = \psi_2 = \mathsf{id}_A$. Hence $\varphi_1[B] = \varphi_2[B]$, i.e., $\m{A}$ does not contain two distinct isomorphic subalgebras. 

For the right-to-left direction, note first that, since $\m{A}$ has the \prp{CEP} and is simple, also every subalgebra is simple and has the \prp{CEP}. In particular, $\hm\sub(\m{A}) = \iso\sub(\m{A}) = \Vfsi$ has the \prp{CEP}, so  \cite[Corollary 2.4]{Fussner2024} yields that $\V$ has the \prp{CEP}, since residuated lattices are congruence-distributive. Now, for every span in $\Vfsi$ we can just take the inclusion embeddings into $\m{A}$ as an amalgam. The assumption that $\m{A}$ does not contain two distinct isomorphic subalgebras ensures that it is actually an amalgam. Hence, by Corollary~\ref{c:ess-AP-sem}, $\V$ has the \prp{AP}.
\end{proof}

Thus, to check if a variety generated by a finite simple totally ordered residuated lattice has the \prp{AP}, it is enough to calculate its subalgebras.

%%%%%%%%%%%%%%%%%%%%%%%%%%%%%%%%%%%%%%%%%%%%%%%%

\section{Idempotent Varieties}\label{sec:idempotent}

In the remainder of the paper, we will discuss what is known about amalgamation in notable varieties of semilinear residuated lattices. We begin with varieties of \emph{idempotent} residuated lattices, i.e., those satisfying the equation $x^2 \approx x$. Idempotent varieties include some of the best-known classes of residuated lattices, such as Boolean algebras, Heyting algebras, and Sugihara monoids. They have also been studied quite thoroughly in recent years, yielding rather powerful structure theorems in several cases; see, e.g., \cite{FG1,FG2,GJM2020,JTV2021}.

For a variety of residuated lattices $\V$, we will denote the subvariety of its idempotent members by $1\V$; for example, $\nSRL{1}$ denotes the variety of idempotent semilinear residuated lattices. Importantly, idempotent semilinear residuated lattices have the \prp{CEP}.

\begin{proposition}[{\cite[Corollary 4.4]{FG1}}]
The variety of idempotent semilinear residuated lattices has the congruence extension property.
\end{proposition}

It was shown in \cite[Theorem~6.6]{GJM2020} that the variety of commutative idempotent semilinear residuated lattices has the \prp{AP}. However, \cite[Theorem~5.2]{FG2} shows that commutativity is crucial for this result; the following example shows that $\nSRL{1}$ does not have the \prp{AP}.

\begin{example}[see {\cite[Theorem~5.2]{FG2}}]\label{ex:IdSRLnoAP}
Consider the totally ordered idempotent residuated lattice $\m{B}$ and $\m{C}$  in Figure~\ref{fig:1SemRL}. One may verify that the span $\tuple{\{\ut\} \hookrightarrow \m{B}, \{\ut\} \hookrightarrow \m{C}}$ does not have a one-sided amalgam among totally ordered idempotent residuated lattices, and hence, by Theorem~\ref{t:APmain}, no amalgam in the variety $\nSRL{1}$.
\end{example}

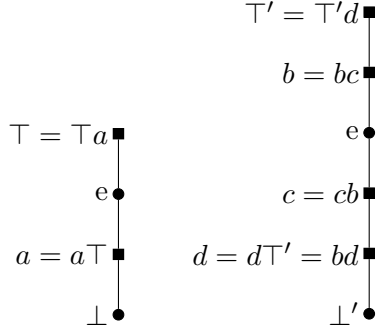
\begin{figure}
\centering
\begin{tikzpicture}[
place/.style={circle,draw=black,fill=black, minimum size = 4pt, inner sep = 0pt},
square/.style={regular polygon,regular polygon sides=4},
place2/.style={square,draw=black,fill=black, minimum size = 5.5pt, inner sep = 0pt}]
   \node[place2] (top) at (0,-1.6) {};
   \node[place] (1) at (0,-2.4) {};
  \node[place2] (a) at (0,-3.2) {};
  \node[place] (bot) at (0,-4) {};
   
  \node[left] () at (top) {$\top = \top a$};
   \node[left] () at (1) {$\ut$};
      \node[left] () at (a) {$a = a\top$};
  \node[left] () at (bot) {$\bot$};

  \draw (top) -- (1) -- (a) -- (bot);
\end{tikzpicture}
\hspace{0.2 in}
\begin{tikzpicture}[
place/.style={circle,draw=black,fill=black, minimum size = 4pt, inner sep = 0pt},
square/.style={regular polygon,regular polygon sides=4},
place2/.style={square,draw=black,fill=black, minimum size = 5.5pt, inner sep = 0pt}]
   \node[place2] (top) at (0,0) {};
   \node[place2] (b) at (0,-0.8) {};
     \node[place] (1) at (-0,-1.6) {};
        \node[place2] (c) at (-0,-2.4) {};
  \node[place2] (d) at (0,-3.2) {};
  \node[place] (bot) at (0, -4) {};
   
  \node[left] () at (top) {$\top' = \top' d$};
   \node[left] () at (b) {$b = bc$};
   \node[left] () at (1) {$\ut$};
   \node[left] () at (c) {$c = cb$};
  \node[left] () at (d) {$d = d\top' = bd$};
    \node[left] () at (bot) {$\bot'$};

  \draw (top) -- (b) -- (1) -- (c) -- (d) -- (bot);
\end{tikzpicture}
\caption{Labeled Hasse diagrams for the totally ordered idempotent residuated lattices  $\m{B}$ (left), and $\m{C}$ (right) that are considered in Example~\ref{ex:IdSRLnoAP}.}
\label{fig:1SemRL}
\end{figure}

We will discuss amalgamation in general subvarieties of $\nSRL{1}$ in Section~\ref{sec:general}. For now, we pause for an interlude in Sections~\ref{sec:godel} and \ref{sec:sugihara} to discuss amalgamation in some well-known, classical cases: Boolean algebras, semilinear Heyting algebras (aka G\"{o}del algebras), and Sugihara monoids.

\subsection{G\"{o}del Algebras and Relative Stone Algebras}\label{sec:godel}

Maksimova's characterization of varieties of Heyting algebras with the amalgamation property \cite{Mak77} is probably the best-known result on amalgamation for residuated lattices. Semilinear Heyting algebras---usually called \emph{G\"{o}del algebras}---figure prominently into this work. Formally, these may be defined as bounded semilinear residuated lattices satisfying $xy\approx x\meet y$. Of course, one consequence of this equation is that $\cdot$ is duplicative and needn't be included explicitly in the signature. Evidently, G\"{o}del algebras are idempotent residuated lattices and further satisfy $x\leq\ut$. We will denote the variety of G\"{o}del algebras by $\GA$.

Like other Heyting algebras, G\"{o}del algebras are bounded. If we drop this requirement, we obtain semilinear algebras in the basic signature of residuated lattices that satisfy $xy\approx x\meet y$. These are called \emph{relative Stone algebras}, and have also been studied quite thoroughly by specialists working on Heyting algebras. We denote the variety of relative Stone algebras by $\RSA$.

There are a few basic examples that are instrumental to the theory of G\"{o}odel algebras and relative Stone algebras. For an integer $m\geq 0$, we define the $m$-element totally ordered relative Stone algebra
\[
\m{R}_{m} = \alg{\{-m+1,\dots, 0\}, \meet,\join, \to, 0}
\]
and the $m$-element totally ordered Gödel algebra 
\[
\m{G}_{m} = \alg{\{-m+1,\dots, 0\}, \meet,\join,\to, 0,-m+1},
\]
where in both cases $\ut = 0$ and
\[
x \to y = \begin{cases}
\ut & \text{if } x \leq y \\
y & \text{otherwise}.
\end{cases}
\]
In fact, it is not hard to see that every totally ordered relative Stone algebra has the form $\alg{C,\meet,\join,  \to \ut}$, where $\alg{C,\meet,\join}$ is a chain with top element $\ut$ and $\to$ is defined as above. Similarly, every totally ordered Gödel algebra is of the form $\alg{C,\meet,\join, \to \ut, \zr}$, where $\alg{C,\meet,\join}$ is a chain with top element $\ut$ and bottom element $\zr$. 

It follows from the previous observation that embeddings between totally ordered relative Stone algebras (Gödel algebras) are exactly order-preserving injections that preserve the top (and bottom). This has drastic consequences for amalgamation in these classes: Amalgamating a span of totally ordered relative Stone algebras or Gödel algebras essentially amounts to amalgamating totally ordered sets, taking care to attend to any designated bounds. From this, it isn't hard to see that $\chain{\RSA}$ and $\chain{\GA}$ have the \prp{AP}. From Corollary~\ref{c:ess-AP-comsem}, it follows that the varieties $\RSA$ and $\GA$ both have the \prp{AP}. 
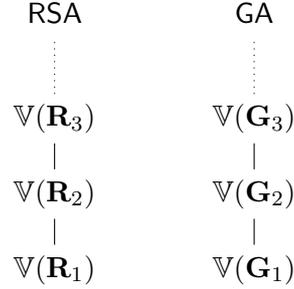
\begin{figure}
\centering
\begin{tikzpicture}
   \node (t) at (0,1) {$\RSA$};
   \node (d) at (0,0.8) {};
   \node (3) at (0,-0.4)  {$\vr(\m{R}_3)$};
  \node (2) at (-0,-1.4)  {$\vr(\m{R}_2)$};
  \node (1) at (0,-2.4) {$\vr(\m{R}_1)$};
  \draw (1) -- (2) -- (3);
  \draw[dotted] (3) -- (d);
\end{tikzpicture}
\quad\quad\quad
\begin{tikzpicture}
   \node (t) at (0,1) {$\GA$};
   \node (d) at (0,0.8) {};
   \node (3) at (0,-0.4)  {$\vr(\m{G}_3)$};
  \node (2) at (-0,-1.4)  {$\vr(\m{G}_2)$};
  \node (1) at (0,-2.4) {$\vr(\m{G}_1)$};
  \draw (1) -- (2) -- (3);
  \draw[dotted] (3) -- (d);
\end{tikzpicture}
\caption{The subvariety lattices of $\RSA$ and $\GA$.}
\label{fig:RSGA-lattice}
\end{figure}

What about amalgamation in subvarieties of $\RSA$ and $\GA$? Maksimova offered a complete description of $\amal(\RSA)$ and $\amal(\GA)$ in the 1970s as part of her study of amalgamation in Heyting algebras generally; see \cite{Mak77}. From the modern perspective, the best starting point is from the well-known fact that $\slat(\RSA)$ and $\slat(\GA)$ are countably infinite chains, as shown in Figure~\ref{fig:RSGA-lattice}. From this fact, the modern tools described in Section~\ref{sec:tools} make it quite easy to get complete descriptions of $\amal(\RSA)$ and $\amal(\GA)$. We illustrate this in the next proposition, focusing on $\amal(\GA)$; $\amal(\RSA)$ is similar.

\begin{proposition}[cf.~\cite{Mak77}]
The non-trivial varieties of Gödel algebras with the amalgamation property are exactly $\vr(\m{G}_2)$, $\vr(\m{G}_3)$, and $\GA$.
\end{proposition}

\begin{proof}
We already concluded that $\GA$ has the amalgamation property. To see that $\vr(\m{G}_2)$ and $\vr(\m{G}_3)$ have the \prp{AP}, note that $\chain{\vr(\m{G}_2)}$ consists up to isomorphism of the algebras $\m{G}_1$ and $\m{G}_2$; $\chain{\vr(\m{G}_3)}$ consists up to isomorphism of the algebras $\m{G}_1,\m{G}_2,$ and $\m{G}_3$. Now, it is straightforward to check that every span of these algebras has an appropriate amalgam.

To see that these are all the non-trivial varieties with the \prp{AP}, suppose for a contradiction that $m\geq 4$ and  $\vr(\m{G}_m)$ has the \prp{AP}. Note that $\chain{\vr(\m{G}_m)}$ consists up to isomorphism of the algebras $\m{G}_1,\m{G}_2,\dots, \m{G}_m$. Consider the following essential span in $\vr(\m{G}_m)$: $\tuple{\varphi_1 \colon \m{G}_3 \to \m{G}_m, \varphi_2 \colon \m{G}_3 \to \m{G}_4}$ with $\varphi_1(-2) = -m+1$, $\varphi_1(-1) = -m+2$, $\varphi_1(0) = 0$; and $\varphi_2(-2) = -3$, $\varphi_2(-1) = -1$, $\varphi_2(0) = 0$.  Then, by Corollary~\ref{c:ess-AP-comsem}, this span has an amalgam $\tuple{\psi_1\colon \m{G}_m\to \m{D}, \psi_2 \colon \m{G}_4 \to \m{D}}$ in $\chain{\vr(\m{G}_m)}$. In particular, in $\m{D}$ we have 
\begin{align*}
\psi_2(-3) < \psi_2(-2) < \psi_2(-1) &= \psi_2(\varphi_2(-1)) \\
												   &= \psi_1(\varphi_1(-1)) \\
												   &= \psi_1(-m+2) < \dots < \psi_1(-1) < \psi_1(0).
\end{align*}
But then $\m{D}$ has at least $m+1$ elements, which contradicts the fact that every totally ordered member of $\chain{\vr(\m{G}_m)}$ has at most $m$ elements.
\end{proof}
Note that $\m{G}_2$ is just the two-element Boolean algebra, and $\vr(\m{G}_2)$ is the variety of all Boolean algebras. Thus, the preceding proposition also covers the case of Boolean algebras.

Essentially the same approach can be used to also describe $\amal(\RSA)$. Note that there is one fewer subvariety with the \prp{AP}, since the bottom element need not be preserved by morphisms in $\RSA$.
\begin{proposition}[cf.~\cite{Mak77}]
The non-trivial varieties of relative Stone algebras with the amalgamation property are exactly $\vr(\m{R}_2)$ and $\RSA$.
\end{proposition}

\subsection{Sugihara Monoids}\label{sec:sugihara}

Sugihara monoids originated in algebraic studies of relevant logic, and have emerged as important ingredients in structural descriptions of residuated lattices generally. Formally, a \emph{Sugihara monoid} is an idempotent commutative semilinear residuated lattice equipped with an additional involutive negation constant $\zr$. A Sugihara monoid is called \emph{odd} if $\zr \approx \ut$, in which case the expansion of the basic signature by the negation constant $\zr$ is, of course, not necessary. We will denote the variety of Sugihara monoids by $\SM$ and the variety of odd sugihara monoids by $\OSM$.

The varieties of Sugihara monoids with the \prp{AP} are completely classified. Indeed, Marchioni and Metcalfe completely described $\amal(\SM)$ by proving quantifier elimination for suitable elementary classes of totally ordered Sugihara monoids and then applying Theorem~\ref{t:QE-AP}; see \cite{MarMet2012}. 

In order to give the classification, we define the following totally ordered Sugihara monoids:
\begin{align*}
\m{S} &= \alg{\Z,\meet,\join,\cdot,\to, 0,0}, \\
\m{S}_{2m} &= \alg{\{-m,\dots,-1,1,\dots,m \}, \meet,\join, \cdot, \to,1,-1} \quad (m\geq 1), \\
\m{S}_{2m+1} &= \alg{\{-m,\dots,-1,0,1,\dots,m \}, \meet,\join, \cdot, \to,0,0} \quad (m\geq 0),
\end{align*}
Here the order is the obvious one induced by the ordered set of integers $\mathbb{Z}$, and the product and residual are defined as follows, where we write $\lvert x \rvert$ for the absolute value of a number:
\[
x \cdot y = \begin{cases}
x \meet y &\text{if } \lvert x \rvert = \lvert y \rvert \\
y              &\text{if } \lvert x \rvert < \lvert y \rvert \\
x              &\text{if } \lvert x \rvert > \lvert y \rvert \\
\end{cases}
\quad
\text{and}
\quad
x \to y = \begin{cases}
(-x) \join y &\text{if } x \leq y \\
(-x) \meet y & \text{otherwise.}
\end{cases}
\]

\begin{theorem}[{\cite[Theorem 4.4]{MarMet2012}}]\label{thm:Sugihara}
The non-trivial varieties of Sugihara monoids with the amalgamation property are exactly \;$\SM$, $\OSM$, $\vr(\m{S}_2)$, $\vr(\m{S}_3)$, $\vr(\m{S}_4)$, $\vr(\{\m{S}_2,\m{S}\})$, $\vr(\{\m{S}_4,\m{S}\})$, and $\vr(\{\m{S}_2,\m{S}_3\})$.
\end{theorem}

\begin{corollary}
The non-trivial varieties of odd Sugihara monoids with the amalgamation property are exactly $\OSM$ and $\vr(\m{S}_3)$
\end{corollary}

An algebraic proof of Theorem~\ref{thm:Sugihara} can be obtained by using the explicit description of $\slat(\SM)$ of \cite{MarMet2012} together with the algebraic techniques discussed in Section~\ref{sec:tools}.

Note that $\m{S}_2$ is yet another presentation for the two-element Boolean algebra, so there is overlap between Theorem~\ref{thm:Sugihara} and the previously mentioned results on G\"{o}del algebras/relative Stone algebras.

Actually, although it is not obvious, there is a deep connection between Sugihara monoids and relative Stone algebras and this can be exploited to obtain alternative proofs of the \prp{AP}. \cite{GalRaf2012} shows that the category of odd Sugihara monoids and homomorphisms is equivalent to the category of relative Stone algebras and homomorphisms, and applies this equivalence to obtain the \prp{AP} for the variety of odd Sugihara monoids. The follow-up to the aforementioned work, \cite{GalRaf2015}, extends the categorical equivalence to the category of arbitrary (not necessarily odd) Sugihara monoids and homomorphisms and the category of relative Stone algebras enriched with an additional unary operation and constant, giving some similar applications. Later on, \cite{FusGal2019} sharpens the equivalence further, showing that the category of Sugihara monoids is equivalent to the category of relative Stone algebras endowed with an additional constant whose upset is a Boolean algebra. Because the \prp{AP} is a categorical property, equivalences of the this kind can be a powerful tool for quickly establishing the \prp{AP}.

\begin{remark}\label{rem:central}
Any cyclic involutive totally ordered idempotent residuated lattice is necessarily commutative, i.e., a Sugihara monoid. Let $\m{A}$ be such an algebra and $\zr$ its involutive negation constant  with involution $\neg$. Then we have 
$\ut \ovr \zr = \neg(\zr \cdot \neg\ut) = \neg(\neg\ut \cdot \zr) = \zr\under \ut$, so $\ut \ovr \zr = \zr \under \ut$. Hence, by \cite[Lemma 3.5]{FG1}, $\zr$ is central. But then also for each $a \in A$, $\ut \ovr a = \neg(a \cdot \zr) = \neg(\zr \cdot a) = a \under \ut$, so, by \cite[Lemma 3.5]{FG1}, every element of $\m{A}$ is central, i.e., $\m{A}$ is commutative.  Thus, there is no direct non-commutative analogue of Sugihara monoids. Note, however, that if we drop the assumption that $\zr$ is a cyclic element, then $\m{A}$ need not be commutative.
\end{remark}

\subsection{Idempotent Varieties in General}\label{sec:general}

Taking stock, we have seen in Sections~\ref{sec:godel} and \ref{sec:sugihara} that we can get a complete understanding of amalgamation in idempotent semilinear varieties when either (1) $\cdot$ coincides with $\meet$, or (2) there is a cyclic involutive negation constant. These assumptions are quite stringent, and their presence imposes a quite rigid structure on totally ordered algebras. Surprisingly, even in the absence of these assumptions, enough of this underlying structure remains intact that we can get a very good understanding of amalgamation in general idempotent semilinear varieties, although the picture is much more complicated.

Notably, among general idempotent semilinear varieties, commutativity has quite an interesting connection with the \prp{AP}. In light of Remark~\ref{rem:central}, each of the previously mentioned assumptions (1) and (2) entails commutativity. Indeed, until recently the absence of examples of non-commutative varieties with the \prp{AP} has been rather conspicuous---so much so that Gil-F\'{e}rez, Ledda, and Tsinakis asked in \cite[Problem~5]{GLT2015} whether there are any such examples at all. This question was answered in \cite{GJM2020}, which exhibits a variety generated by a single strictly simple finite non-commutative totally ordered idempotent residuated lattice that has the \prp{AP} (see Example~\ref{ex:strictly-simple}). Quite surprisingly, it turns out that commutativity is actually an \emph{obstacle} to an idempotent semilinear variety having the \prp{AP}, in the sense illustrated by the results below. 

One of the keys to understanding amalgamation in general idempotent semilinear varieties is the \emph{nested sum}. The latter is just one of several constructions in the literature that focus on constructing new residuated lattices from old ones by `replacing' the unit $\ut$ of a given residuated lattice by another residuated lattice. Constructions of this flavor have proven very effective in studies of amalgamation, since they allow for a `component-wise' amalgamation of a span, and we will see them again in Section~\ref{sec:cancellative}. 

Due to their importance in amalgamation, we pause to give the technical details of the nested sum construction. First, a totally ordered residuated lattice $\m{A}$ is called \emph{admissible} if $a\under \ut, \ut \ovr a \notin \{\ut\}$ for each $a \in A{\setminus}\{\ut\}$. Given a totally ordered set $\alg{I,\leq}$ whose top element (if it exists) is denoted by  $\top$, we say that an indexed family $(\m{A}_i)_{i\in I}$ is \emph{admissible} if for each $i\in I{\setminus}\{\top\}$, $\m{A}_i$ is admissible. Let $\alg{I,\leq}$ be a non-empty chain, and $(\m{A}_i)_{i\in I}$ and admissible family of totally ordered residuated lattices $\alg{A_i,\meet_i,\join_i,\cdot_i,\under_i,\ovr_i,\ut}$, where we assume that $A_i \cap A_j = \{\ut \}$  for $i\neq j$. The \emph{nested sum} $\nsum_{i\in I} \m{A}_i$ is the totally ordered residuated lattice with universe $A = \bigcup_{i\in I} A_i$ that is defined as follows. Its order  $\leq$ is the smallest partial order on $A$ satisfying
\begin{enumerate}
\item ${\leq_i} \subseteq {\leq}$ for each $i\in I$;
\item if $i<j$, $x\in A_i$, $y\in A_j$, and $x<_i \ut$, then $x \leq y$;
\item if $i<j$, $x\in A_i$, $y\in A_j$, and $\ut<_i  x$, then $y \leq x$.
\end{enumerate}
Moreover for $\ast \in \{\cdot, \under, \ovr\}$, we define $x \ast y = x \ast_i y$ if $x,y \in A_i$, and for $x \in A_i$, $y \in A_j$ with $i<j$, we let $x \ast y = x \ast_i \ut$ and $y \ast x = \ut \ast_i x$. If $\alg{ I, \leq}$ is the $n$-element chain $\alg{\{1,\dots, n\}, \leq}$ we also write $\nsum_{i=1}^n \m{A}_i$ for $\nsum_{I} \m{A}_i$. In particular if $n=2$, we write $\m{A}_1 \nsum \m{A}_2$.

As explained in \cite{FG1,FG2}, there are a few term operations that turn out to be quite significant in the study of totally ordered idempotent residuated lattices. We define $x^\ell = \ut \ovr x$, $x^r = \ut \under x$, and $x^{\starm} = x^\ell \meet x^r$.
A idempotent semilinear  residuated lattice $\m{A}$ is called \emph{$^{\starm}$-involutive} or \emph{lower involutive} if it satisfies $x^{{\starm}{\starm}} = x$. Lower involutive totally ordered idempotent residuated lattices play an important role in the general structure theory of totally ordered idempotent residuated lattices, and they admit the following tidy description.

\begin{lemma}[{\cite[Lemma~4.22]{FG2}}]\label{lem:inv decomp}
Every lower involutive totally ordered idempotent residuated lattice is the nested sum of its $1$-generated subalgebras.
\end{lemma}

The previous lemma is used in \cite{FG2} to show that the variety of lower involutive idempotent semilinear residuated lattices has the \prp{AP}.

In the presence of commutativity, a totally ordered idempotent residuated lattice is lower involutive if and only if it satisfies $(x\to\ut)\to\ut \approx x$, i.e., if and only if it is an odd Sugihara monoid. Thus, in the commutative case, the preceding lemma states that every totally ordered odd Sugihara monoid is the nested sum of copies of the three-element Sugihara monoid.

Note that \cite{Galatos2004} gives an uncountable family of subvarieties of $\nSRL{1}$, each of which contains non-commutative algebras. Through an analysis of these varieties using Lemma~\ref{lem:inv decomp} and Theorem~\ref{t:APmain}, \cite{FMS2023} exhibits continuum-many subvarieties of idempotent semilinear residuated lattices with the \prp{AP}. Hence:

\begin{theorem}[{\cite[Theorem~A]{FMS2023}}]\label{thm:idem noncom}
$\amal(\nSRL{1})$ has cardinality of the continuum. In fact, there are continuum-many varieties of idempotent semilinear residuated lattices that have the amalgamation property and contain non-commutative members.
\end{theorem}

Nested sum decompositions are also useful in the commutative case. To take full advantage of them, we first introduce a family of  commutative totally ordered idempotent residuated lattices that will constitute building blocks for some nested sum decompositions. For $m,n \in \N$, we define the totally ordered idempotent residuated lattice $\Com{m}{n} = \alg{\com{m}{n}, \meet, \join, \cdot, \under, \ovr,\ut}$ with $\com{m}{n} = \{b_m < \dots < b_0 < \ut < a_n < \dots < a_0\}$, where $a_i \cdot a_j = a_{\min(i,j)} = a_i \join a_j$, $b_k \cdot b_l = b_{\max(k,l)} = b_k \meet b_l$, and $a_i\cdot b_k = b_k \cdot  a_i = b_k$. Moreover, we recall that $\m{R}_p$ denotes the $p$-element totally ordered relative Stone algebra.

\begin{lemma}[{\cite[Lemma~3.10]{FMS2023}}]\label{lem:com decomp}
Every finite totally ordered idempotent commutative residuated lattice is isomorphic to a nested sum of the form $(\nsum_{i=1}^k \Com{m_i}{n_i}) \boxplus \m{R}_p$.
\end{lemma}

In the non-commutative case, even $1$-generated totally ordered idempotent residuated lattices may be infinite. However, in the presence of commutativity, idempotent semilinear residuated lattices are locally finite, as shown in \cite{GJM2020}. Using this fact, in conjunction with Lemma~\ref{lem:com decomp} and Theorem~\ref{t:APmain}, \cite{FMS2023} shows that there are only $60$ varieties of commutative idempotent semilinear residuated lattices with the \prp{AP}.

\begin{theorem}[{\cite[Theorem~B]{FMS2023}}]\label{thm:idem com}
$|\amal(\nCSRL{1})| = 60$. That is, there are exactly 60 varieties of commutative idempotent semilinear residuated lattices with the \prp{AP}.
\end{theorem}

In fact, \cite[Section~5]{FMS2023} explicitly describes $\amal(\nCSRL{1})$. Further, even with the addition of an extra constant to the signature (which a priori need satisfy no assumptions), there are still finitely many commutative idempotent semilinear varieties with the \prp{AP}.

Taken together, Theorems~\ref{thm:idem noncom} and \ref{thm:idem com} paint a rather unexpected picture of commutativity's role in amalgamation: Although only a decade ago it was unclear whether there are any non-commutative varieties with the \prp{AP} at all \cite{GLT2015}, it is now clear that, at least in some contexts, commutativity may restrict the scope of the \prp{AP}. This raises the following open question:

\begin{question}
Are there uncountably many varieties of commutative semilinear residuated lattices with the amalgamation property?
\end{question}

Note that, without the restriction to semilinear varieties, this question is resolved: It is shown in \cite{FS23} that there are continuum-many varieties of commutative residuated lattices with the \prp{AP}. However, most of the members of the varieties constructed in \cite{FS23} are not even distributive.

As a final remark regarding idempotent semilinear varieties, we note that techniques similar to those referenced in this paper have also been successful in studying the \prp{AP} in adjacent, non-residuated algebras. In \cite{Santschi2024}, nested sum decompositions are used in the context of semilinear idempotent distributive $\ell$-monoids to identify most of the varieties of the latter with the \prp{AP}.

%%%%%%%%%%%%%%%%%%%%%%%%%%%%%%%%%%%%%%%%%%%%%%%%%%%%

\section{Dropping Idempotence: Knotted Varieties and More}\label{sec:knotted}

As Section~\ref{sec:idempotent} attests, we now have a rather sharp picture of how amalgamation works in idempotent semilinear varieties. In this section, we will see what happens when idempotence is weakened or dropped entirely. Most notably, we give an elementary proof that the \prp{AP} fails for the variety of all commutative semilinear residuated lattices, answering the question posed in \cite[Problem~8]{GLT2015} and \cite[p.~205]{MPT23}. In fact, our proof of this applies to a large number of varieties lying in between $\SemRL$ and the variety of idempotent semilinear residuated lattices, as well as their expansions by negation constants and bounds.

Our discussion is best expressed in reference to \emph{knotted inequalities}, which have the form $x^m\leq x^n$ for given natural numbers $n,m$. Idempotence is equivalent to the conjunction of the two knotted inequalities $x\leq x^2$, often called the \emph{square increasing} property, and $x^2\leq x$, often called the \emph{square decreasing} property. More generally, a residuated lattice is called \emph{$\langle m,n\rangle$-periodic} if it satisfies both of the knotted inequalities $x^m\leq x^n$ and $x^n\leq x^m$, i.e., if it satisfies $x^m\approx x^n$. An \emph{$n$-potent} residuated lattice is an $\langle n+1,n\rangle$-periodic one, i.e., one that satisfies $x^{n+1}\approx x^n$. Note that if a residuated lattice $\m{A}$ is totally ordered and $\langle m,n \rangle$-periodic for $m>n$, then it is $n$-potent. For let $a \in A$. If $a \leq \ut$, then $a^m \leq a^{n+1} \leq a^n$; and if $\ut \leq a$, then $a^n \leq a^{n+1} \leq a^m$. Hence, since $\m{A}$ is totally ordered, $a^n = a^{n+1}$. So for semilinear varieties the notion of $\langle m,n\rangle$-periodicity reduces to the notion of $n$-potence.

Residuated lattices satisfying the knotted inequality $x\leq x^0 = \ut$ are usually called \emph{integral}. On the other hand, a residuated lattice that satisfies a knotted inequality $\ut\leq x^n$ for $n\neq 0$ is necessarily the one-element algebra. 

The following lemma is sometimes quite handy for showing the failure of \prp{AP} in concrete cases. Its proof is elementary and we omit it.

\begin{lemma}\label{lem:handy}
Let $\m{A}$ be a residuated lattice and suppose that $x,y,f \in A$ are such that $x\under f = x$ and $y\under f = y$. Then $x = y$, or $x$ and $y$ are incomparable. In particular, if $\m{A}$ is totally ordered, then for each $f\in A$, the map $x \mapsto x \under f$ has at most one fixed point.
\end{lemma}

\begin{theorem}\label{thm:knotted}
Let $k\geq 2$, and let $m\geq 1$, $n\geq 0$ with $m\neq n$. Suppose $E$ is any subset of the set of equations $\{xy\approx yx, x^k\approx x^{k+1}, x^m\leq x^n\}$. Then the subvariety $\V$ of $\SemRL$ defined by $E$ does not have the \prp{AP}. The same remains true for the varieties obtained by expanding $\V$ by bounds and/or a negation constant satisfying any of cyclicity, left- or right-involutivity, or involutivity. In particular, $\SemRL$ and $\CSemRL$ do not have the \prp{AP}.
\end{theorem}

\begin{proof}
We consider two spans, showing that neither of them has any amalgam among semilinear residuated lattices.

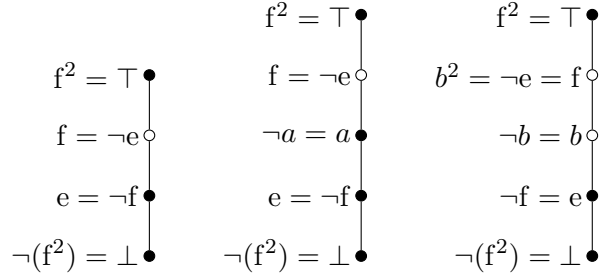
\begin{figure}[t]
\centering
\begin{tikzpicture}[
place/.style={circle,draw=black,fill=black, minimum size = 4pt, inner sep = 0pt},
square/.style={regular polygon,regular polygon sides=4},
place2/.style={square,draw=black,fill=black, minimum size = 5.5pt, inner sep = 0pt},
place3/.style={square,draw=black, minimum size = 5.5pt, inner sep = 0pt},
place4/.style={circle,draw=black, minimum size = 4pt, inner sep = 0pt}]
   \node[place] (top) at (0,0) {};
   \node[place4] (f) at (0,-0.8) {};
  \node[place] (t) at (0,-1.6) {};
   \node[place] (bot) at (0,-2.4) {};
   
  \node[left] () at (top) {$\zr^2=\top$};
   \node[left] () at (f) {$\zr=\neg\ut$};
  \node[left] () at (t) {$\ut=\neg\zr$};
  \node[left] () at (bot) {$\neg (\zr^2) = \bot$};

  \draw (bot) -- (t) -- (f) -- (top);
\end{tikzpicture}
\hspace{0.2 in}
\begin{tikzpicture}[
place/.style={circle,draw=black,fill=black, minimum size = 4pt, inner sep = 0pt},
square/.style={regular polygon,regular polygon sides=4},
place2/.style={square,draw=black,fill=black, minimum size = 5.5pt, inner sep = 0pt},
place3/.style={square,draw=black, minimum size = 5.5pt, inner sep = 0pt},
place4/.style={circle,draw=black, minimum size = 4pt, inner sep = 0pt}]
   \node[place] (top) at (0,0) {};
   \node[place4] (f) at (0,-0.8) {};
  \node[place] (a) at (-0,-1.6) {};
  \node[place] (t) at (0,-2.4) {};
   \node[place] (bot) at (0,-3.2) {};
   
  \node[left] () at (top) {$\zr^2=\top$};
   \node[left] () at (f) {$\zr=\neg\ut$};
      \node[left] () at (a) {$\neg a = a$};
  \node[left] () at (t) {$\ut=\neg\zr$};
  \node[left] () at (bot) {$\neg(\zr^2) = \bot$};

  \draw (bot) -- (t) -- (a) -- (f) -- (top);
\end{tikzpicture}
\hspace{0.2 in}
\begin{tikzpicture}[
place/.style={circle,draw=black,fill=black, minimum size = 4pt, inner sep = 0pt},
square/.style={regular polygon,regular polygon sides=4},
place2/.style={square,draw=black,fill=black, minimum size = 5.5pt, inner sep = 0pt},
place3/.style={square,draw=black, minimum size = 5.5pt, inner sep = 0pt},
place4/.style={circle,draw=black, minimum size = 4pt, inner sep = 0pt}]
   \node[place] (top) at (0,0) {};
   \node[place4] (f) at (0,-0.8) {};
  \node[place4] (b) at (-0,-1.6) {};
  \node[place] (t) at (0,-2.4) {};
   \node[place] (bot) at (0,-3.2) {};
   
  \node[left] () at (top) {$\zr^2=\top$};
   \node[left] () at (f) {$b^2=\neg\ut=\zr$};
      \node[left] () at (b) {$\neg b = b$};
  \node[left] () at (t) {$\neg\zr=\ut$};
  \node[left] () at (bot) {$\neg(\zr^2)= \bot$};

  \draw (bot) -- (t) -- (b) -- (f) -- (top);
\end{tikzpicture}

\caption{Labeled Hasse diagrams for the commutative, involutive, square-increasing residuated lattices $\m{A}_1$ (left), $\m{B}_1$ (middle), and $\m{C}_1$ (right) discussed in the proof of Theorem~\ref{thm:knotted}. Note that $a$ and $b$ are both fixed points under the negation $x\mapsto x\to \zr$}
\label{fig:SDMM}
\end{figure}

First, consider Figure~\ref{fig:SDMM}. The information depicted in the labeled Hasse diagrams therein is sufficient to uniquely determine the three commutative, square-increasing residuated lattices $\m{A}_1$, $\m{B}_1$, and $\m{C}_1$, where $\zr$ may be taken as an involutive negation constant in each case. It is easy to see that these algebras are $k$-potent for all $k\geq 3$ and that they satisfy every knotted inequality $x^m\leq x^n$ for $2\leq n \leq m$ or $1\leq m \leq n$. Being commutative and involutive, they are trivially cyclic and left- and right-involutive.

The inclusion of $\m{A}_1$ into $\m{B}_1$ and $\m{C}_1$ determines an essential span $\alg{\m{A}_1\hookrightarrow\m{B}_1,\m{A}_1\hookrightarrow\m{C}_1}$, and we claim that this essential span has no semilinear amalgam. To see this, by Lemma~\ref{l:doubly-essential}, it is enough to show that it does not have an amalgam in the class of totally ordered residuated lattices. Suppose for a contradiction that $\langle \psi_1 \colon \m{B}_1 \to \m{D}, \psi_2 \colon \m{C}_1 \to \m{D} \rangle$ is an amalgam such that $\m{D}$ is totally ordered. Note that both $a$ and $b$ are fixed points of the map $x\mapsto x\to\zr$ in the respective algebra, so $\psi_1(a)=\psi_2(b)$ by Lemma~\ref{lem:handy}. However, this implies $\psi_1(a) = \psi_1(a)^2 = \psi_2(b)^2 = \psi_2(\zr)=\psi_1(\zr)$, contradicting the fact that $\psi_1$ is an injection.

\begin{figure}[t]
\centering
\begin{tikzpicture}[
place/.style={circle,draw=black,fill=black, minimum size = 4pt, inner sep = 0pt},
square/.style={regular polygon,regular polygon sides=4},
place2/.style={square,draw=black,fill=black, minimum size = 5.5pt, inner sep = 0pt},
place3/.style={square,draw=black, minimum size = 5.5pt, inner sep = 0pt},
place4/.style={circle,draw=black, minimum size = 4pt, inner sep = 0pt}]
   \node[place] (t) at (0,0) {};
   \node[place] (a) at (0,-0.8) {};
  \node[place] (b) at (0,-1.6) {};
     \node[place4] (b') at (0,-2.4) {};
   \node[place4] (a') at (0,-3.2) {};
  \node[place] (t') at (0,-4) {};
  
  \node[left] () at (t) {$\ut$};
    \node[left] () at (a) {$a$};
  \node[left] () at (b) {$ab = b$};
    \node[left] () at (b') {$\neg b$};
    \node[left] () at (a') {$\neg a$};
  \node[left] () at (t') {$\zr = \neg\ut$};

  \draw (t') -- (a') -- (b') -- (b) -- (a) -- (t);
\end{tikzpicture}
\hspace{0.2 in}
\begin{tikzpicture}[
place/.style={circle,draw=black,fill=black, minimum size = 4pt, inner sep = 0pt},
square/.style={regular polygon,regular polygon sides=4},
place2/.style={square,draw=black,fill=black, minimum size = 5.5pt, inner sep = 0pt},
place3/.style={square,draw=black, minimum size = 5.5pt, inner sep = 0pt},
place4/.style={circle,draw=black, minimum size = 4pt, inner sep = 0pt}]
   \node[place] (t) at (0,0) {};
   \node[place] (a) at (0,-0.8) {};
  \node[place4] (x) at (-0,-1.6) {};
  \node[place] (b) at (0,-2.4) {};
     \node[place4] (b') at (0,-3.2) {};
   \node[place4] (x') at (0,-4) {};
  \node[place4] (a') at (-0,-4.8) {};
  \node[place] (t') at (0,-5.6) {};
   
  \node[left] () at (t) {$\ut$};
   \node[left] () at (a) {$a$};
      \node[left] () at (x) {$ax=x$};
  \node[left] () at (b) {$x^2=b$};
    \node[left] () at (b') {$\neg b$};
   \node[left] () at (x') {$\neg x$};
      \node[left] () at (a') {$\neg a$};
  \node[left] () at (t') {$\zr = \neg\ut$};

  \draw (t') -- (a') -- (x') -- (b') -- (b) -- (x) -- (a) -- (t);
\end{tikzpicture}
\hspace{0.2 in}
\begin{tikzpicture}[
place/.style={circle,draw=black,fill=black, minimum size = 4pt, inner sep = 0pt},
square/.style={regular polygon,regular polygon sides=4},
place2/.style={square,draw=black,fill=black, minimum size = 5.5pt, inner sep = 0pt},
place3/.style={square,draw=black, minimum size = 5.5pt, inner sep = 0pt},
place4/.style={circle,draw=black, minimum size = 4pt, inner sep = 0pt}]
   \node[place] (t) at (0,0) {};
   \node[place] (a) at (0,-0.8) {};
     \node[place4] (y) at (-0,-1.6) {};
        \node[place4] (z) at (-0,-2.4) {};
  \node[place] (b) at (0,-3.2) {};
   \node[place4] (b') at (0,-4) {};
   \node[place4] (z') at (0,-4.8) {};
     \node[place4] (y') at (-0,-5.6) {};
        \node[place4] (a') at (-0,-6.4) {};
  \node[place] (t') at (0,-7.2) {};
   
  \node[left] () at (t) {$\ut$};
   \node[left] () at (a) {$a$};
   \node[left] () at (y) {$y$};
   \node[left] () at (z) {$az=ay=z$};
  \node[left] () at (b) {$y^2=z^2 = b$};
    \node[left] () at (b') {$\neg b$};
   \node[left] () at (z') {$\neg z$};
   \node[left] () at (y') {$\neg y$};
   \node[left] () at (a') {$\neg a$};
  \node[left] () at (t') {$\zr =\neg \ut$};

  \draw (t') -- (a') -- (y') -- (z') -- (b') -- (b) -- (z) -- (y) -- (a) -- (t);
\end{tikzpicture}
\caption{Labeled Hasse diagrams for the algebras $\m{A}_2$ (left), $\m{B}_2$ (middle), and $\m{C}_2$ (right) discussed in the proof of Theorem~\ref{thm:knotted}. Note that the interval $[b,e]$ forms a residuated lattice in each of the depicted algebras and  multiplication between the remaining elements is defined according to the following rules: if $c,d \in [b,e]$, then $c(\neg d) = \neg (c \to d)$ and $(\neg c)(\neg d) = \zr$.} 
\label{fig:MTL}
\end{figure}
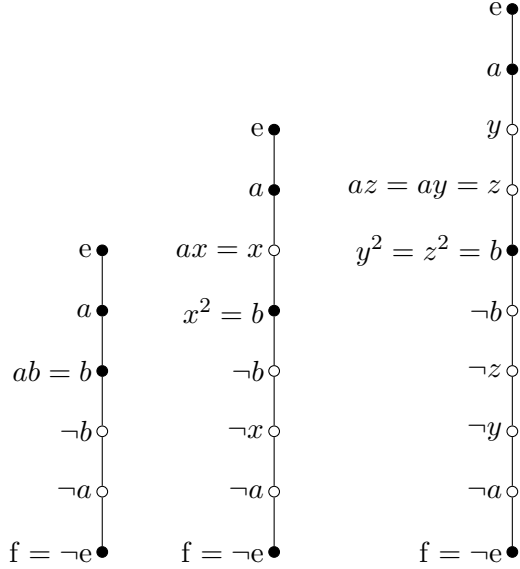

This proves the theorem for all cases except where $k=2$ or $1=n<  m$. For the remaining cases, we refer to Figure~\ref{fig:MTL}.\footnote{A similar example, showing the failure of \prp{AP} in the variety of MTL-algebras and a number of other varieties, was recently presented by V.~Giustarini and S.~Ugolini at a seminar at Chapman University.} The labeled Hasse diagrams therein uniquely determine the three commutative, integral, involutive residuated lattices $\m{A}_2$, $\m{B}_2$, and $\m{C}_2$. We claim once again that the essential span defined by $\langle \m{A}_2\hookrightarrow\m{B}_2,\m{A}_2\hookrightarrow\m{C}_2\rangle$ has no amalgam among semilinear residuated lattices. To see this, by Lemma~\ref{l:doubly-essential}, it is enough to show that it does not have an amalgam in the class of totally ordered residuated lattices.

Toward a contradiction, suppose that $\langle \psi_1\colon \m{B}_2 \to \m{D}, \psi_2\colon \m{C}_2 \to \m{D} \rangle$ is an amalgam of the span with $\m{D}$ a totally ordered residuated lattices. Observe that in $\m{B}_2$ we have that $x\to b = x$, and in $\m{C}_2$ we have $y\to b= y$. Hence, in $\m{D}$ each of $\psi_1(x)$ and $\psi_2(y)$ are fixed points of the map $p\mapsto p\to \psi_1(b) = p\to \psi_2(b)$. It follows from Lemma~\ref{lem:handy} that $\psi_1(x)=\psi_2(y)$. But then $\psi_2(z) = \psi_2(a)\psi_2(y)=\psi_1(a)\psi_1(x)=\psi_1(x)=\psi_2(y)$, contradicting the fact that $\psi_2$ is an injection.\end{proof}

Theorem~\ref{thm:knotted} is quite expansive, and it is worth pausing to consider its consequences. Firstly, Theorem~\ref{thm:knotted} shows the failure of the \prp{AP} in the variety of semilinear residuated lattices, the variety of commutative semilinear residuated lattices, the variety of integral semilinear residuated lattices, and the variety of commutative integral semilinear residuated lattices, as well as their expansions by bounds and/or involution. This solves several of the open problems listed in \cite[p.~205]{MPT23}. Note that the failure of the \prp{AP} for the variety of semilinear residuated lattices was already shown in  \cite{GLT2015} (see also Section~\ref{sec:cancellative}).

Secondly, we may contrast Theorem~\ref{thm:knotted} against the results we obtained in the idempotent case in Section~\ref{sec:idempotent}. We have already seen that the \prp{AP} holds for the variety of idempotent commutative semilinear residuated lattices, idempotent integral semilinear residuated lattices (Brouwerian algebras), and idempotent involutive commutative semilinear residuated lattices (Sugihara monoids). Replacing idempotence by any of the obvious weaker conditions---viz. knotted inequalities, $n$-potence, or, more generally, $\langle m,n\rangle$-periodicity---gives a variety for which the \prp{AP} fails in any of these cases.

Thirdly, Theorem~\ref{thm:knotted} contributes to the study of amalgamation in varieties of \emph{De Morgan monoids}. The latter are commutative, distributive, involutive, square-increasing residuated lattices, and have gained prominence because of their connection to relevance logic; see, e.g., \cite{MRW2020}. Notably, amalgamation is known to fail for the variety of De Morgan monoids as well a great number of adjacent varieties; see \cite{Urq1993}. Theorem~\ref{thm:knotted} shows that the \prp{AP} also fails for the variety of semilinear De Morgan monoids $\pc{SDMM}$, which has recently been studied in \cite{WR2024}. However, despite the rather stark failure of the \prp{AP} in the vicinity of De Morgan monoids, the next proposition shows that $\amal(\pc{SDMM})$ is infinite.

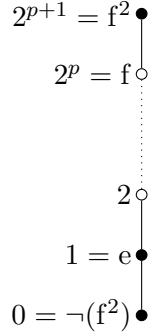
\begin{figure}
\centering
\begin{tikzpicture}[
place/.style={circle,draw=black,fill=black, minimum size = 4pt, inner sep = 0pt},
square/.style={regular polygon,regular polygon sides=4},
place2/.style={square,draw=black,fill=black, minimum size = 5.5pt, inner sep = 0pt},
place3/.style={square,draw=black, minimum size = 5.5pt, inner sep = 0pt},
place4/.style={circle,draw=black, minimum size = 4pt, inner sep = 0pt}]
   \node[place] (top) at (0,0) {};
   \node[place4] (f) at (0,-0.8) {};
   \node[place4] (m) at (0,-2.4) {};
  \node[place] (t) at (0,-3.2) {};
   \node[place] (bot) at (0,-4) {};
   
  \node[left] () at (top) {$2^{p+1}=\zr^2$};
   \node[left] () at (f) {$2^p=\zr$};
   \node[left] () at (m) {$2$};   
  \node[left] () at (t) {$1=\ut$};
  \node[left] () at (bot) {$0= \neg (\zr^2)$};

  \draw[dotted] (m) -- (f);
  \draw (bot) -- (t) -- (m);
  \draw (f) -- (top);
\end{tikzpicture}
\caption{Labeled Hasse diagrams for the De Morgan monoid $\m{M}_p$.}
\label{fig:DMM2}
\end{figure}

\begin{proposition}
There are infinitely many varieties of semilinear De Morgan monoids with the \prp{AP}.
\end{proposition}

\begin{proof}
For each prime number $p$, we may define a totally ordered De Morgan monoid $\m{M}_p$ whose universe is $M_p=\{0,1,2,\ldots,2^p,2^{p+1}\}$ as depicted in Figure~\ref{fig:DMM2}. The product of $\m{M}_p$ is defined as the usual product of integers, but truncated at the maximum value $2^{p+1}$. It is shown in \cite{MRW2020} that $\vr(\m{M}_p)\neq\vr(\m{M}_q)$ for $p\neq q$, and that the only non-trivial subalgebra of $\m{M}_p$ is the one with universe $\{0,1,2^p,2^{p+1}\}$. In particular, $\m{M}_p$ is simple for each prime number $p$. Thus, by Proposition~\ref{prop:simple chain}, $\vr(\m{M}_p)$ has the amalgamation property for each prime number $p$.
\end{proof}

The previous proposition is still far from a classification. Although the structure of semilinear De Morgan monoids is rather complicated, we conjecture that a transparent and explicit description of $\amal(\pc{SDMM})$ is possible.

\begin{question}
Is it possible to completely describe $\amal(\pc{SDMM})$?
\end{question}

We saw in Section~\ref{sec:idempotent} that a crucial part in characterizing varieties with the \prp{AP} is having good structure theorems. However, beyond the idempotent case, such structure theorems are only in their beginnings. A notable structure theorem for odd or even involutive totally ordered commutative residuated lattices is obtained in \cite{Jenei2022} and applied in \cite{Jenei2023} to obtain some results about the \prp{AP} for certain classes of involutive commutative semilinear  residuated lattices.  We also want to mention the paper \cite{GalatosUgolini2023} where gluings of residuated lattices are considered and applied to show that a rather special variety of $2$-potent integral commutative semilinear residuated lattices fails the \prp{AP}.

\section{Cancellative Varieties and Their Cousins}\label{sec:cancellative}

We now have a rather sophisticated understanding of amalgamation in the vicinity of idempotent semilinear residuated lattices and their obvious generalization. On the other hand, multiplication exhibits rather the opposite behavior of idempotence in \emph{cancellative} residuated lattices, which are those that satisfy the quasiequations
\[
xy \approx xz \Rightarrow y \approx z \text{ and } yx \approx zx \Rightarrow y \approx z,
\]
or, equivalently, the equations
\[
x \under xy \approx y \text{ and } yx \ovr x \approx y.
\]
In any residuated lattices that is both idempotent and cancellative, we have $x\cdot x = x =x\cdot\ut$ and thus $x=\ut$. Hence, the only residuated lattice that is both cancellative and idempotent is the one-element algebra. Information about amalgamation in the vicinity of cancellative varieties therefore complements the theory we have developed in the idempotent case, giving us some sense of the picture as a whole.

Unfortunately, we will see that the \prp{AP} is much less well understood for varieties close to cancellative ones than it is for varieties close to idempotent ones. We will report here on what is known about the \prp{AP} for prominent classes of cancellative semilinear residuated lattices, as well as those constructed from them---notably Wajsberg hoops, MV-algebras, basic hoops, and BL-algebras. However, many difficult open questions remain.

In what follows, we denote the variety of cancellative semilinear residuated lattices by $\CanSRL$ and the subvariety consisting of commutative members of $\CanSRL$ by $\CCanSRL$.

\subsection{Lattice-ordered Groups}\label{sec:lgroups}

Much of what we know about cancellative residuated lattices descends from the extensive body of work on lattice-ordered groups, which will serve as our starting point. A \emph{lattice-ordered group} (or \emph{$\ell$-group} for short) is traditionally defined as an algebra $\m{A} = \alg{A,\meet,\join,\cdot,{}^{-1},\ut}$ such that 
\begin{itemize}
\item $\alg{A,\meet,\join}$ is a lattice;
\item $\alg{A,\cdot, {}^{-1}}$ is a group;
\item for all $a,b,c,d \in A$, 
\[
a(b \join c)d = abd \join acd.
\]
\end{itemize}
Clearly, the class of a $\ell$-groups forms a variety, which we denote by $\LG$. Note that $\ell$-groups are term equivalent to residuated lattices that satisfy $(\ut \ovr x)x \approx \ut$: In any $\ell$-group, we may define the requisite residual operations by $x\under y := x^{-1}y$ and $x\ovr y := xy^{-1}$; conversely, in any residuated lattice satisfying the aforementioned equation, $x^{-1} := \ut \ovr x$ defines a group inverse operation. In light of this term equivalence, we will toggle between these presentations of $\ell$-groups as convenient.

Thanks to their long history, $\ell$-groups come with terminology that diverges from that usually used for residuated lattices: Semilinear $\ell$-groups are traditionally called \emph{representable}, and commutative $\ell$-groups are traditionally called \emph{abelian}. We will denote the variety of representable $\ell$-groups by $\RLG$ and the variety of abelian $\ell$-groups by $\ALG$. The variety $\ALG$ is the unique atom in the subvariety lattice $\slat(\LG)$ (see, e.g., \cite{KM94}). In particular, every abelian $\ell$-group is semilinear.

The variety of abelian $\ell$-groups is well-known to have the \prp{AP}. The first proof of this fact is due to Pierce in \cite{Pierce1972b}. An alternative proof by Pierce using the Hahn embedding theorem (see \cite{Hahn1907}) can be found in \cite{Pierce1972a}. Two algebraic proofs of Powell and Tsinakis can be found in \cite{PowellTsinakis1983,PowellTsinakis1989b}, and a model-theoretic proof via quantifier elimination, 
due to Weispfenning, can be found in \cite{Weispfenning1989}. The recent proof by Metcalfe, Montagna, and Tsinakis in \cite{Metcalfe2014} proceeds by showing that $\ALG$ has the equational deductive interpolation property, which, for varieties of commutative residuated lattices, entails the \prp{AP}.

Lexicographic products, which are central to the Hahn embedding theorem, play a role in Pierce's proof in \cite{Pierce1972a} that is analogous the role nested sums played in Section~\ref{sec:general}; ordinal sums will play much the same part in Section~\ref{sec:BL}. However, in the following proof that $\ALG$ has the \prp{AP}, we opt to follow the model-theoretic treatment of \cite{Weispfenning1989}, modulo a slight change in the language. In our opinion, it is the shortest and least complicated of the available proofs, and also affords the opportunity to illustrate model-theoretic techniques.

For our discussion of abelian $\ell$-groups, it is convenient to use additive notation rather than multiplicative notation. Consequently, we will also write $x^n$ additively as $nx$ for $n\in \N$.
 Recall that an abelian group $\m{G} = \alg{G, +, -, 0}$ is called \emph{divisible} if for each $a \in G$ and $n\geq 1$, there exists $b\in G$ such that $nb = a$.

\begin{proposition}\label{p:divisible-QE}
The class of non-trivial totally ordered divisible abelian groups (considered in the language $\{<,+,-,0\}$) has quantifier elimination. 
\end{proposition}

\begin{proof}[Proof sketch]
We follow the strategy of the proof in \cite{Weispfenning1989}.
Let $\mathsf{D}$ be the class of non-trivial divisible totally ordered abelian groups.
To prove quantifier elimination for $\mathsf{D}$ it suffices to show that every formula of the form $(\exists y)\psi(y,x_1,\dots,x_n)$ is equivalent in  $\mathsf{D}$ to a quantifier-free formula, where
\begin{align*}
\psi(y,x_1,\dots, x_n) &=  \bigcurlywedge_{i\in I}  ((\sum_{j=1}^n \lambda_{i,j} x_j)  \approx y) \curlywedge \bigcurlywedge_{k\in K} ((\sum_{j=1}^n \lambda'_{k,j} x_j) < y) \\
& \quad \curlywedge  \bigcurlywedge_{l\in L} ( y < (\sum_{j=1}^n \lambda''_{l,j} x_j))   \curlywedge \chi(x_1,\dots, x_n).
\end{align*}
for $\lambda_{i,j}, \lambda'_{k,j} \in \Z$ and $I,K,L$ finite sets.

If  $I\neq \emptyset$, then for $i\in I$, we define $t = \sum_{j=1}^n \lambda_{i,j} x_j$ and we get
\[
\mathsf{D} \models (\forall \overline{x})[(\exists y)\psi(y,\overline{x}) \Leftrightarrow \psi(t, \overline{x})],
\]
If $I = K =  \emptyset$ or $I = L = \emptyset$, then 
\[
\mathsf{D} \models  (\forall \overline{x})[(\exists y)\psi(y,\overline{x}) \Leftrightarrow \chi(\overline{x})]
\]
Finally, if $I = \emptyset$, $K\neq \emptyset$, and $L \neq \emptyset$, then we define for $k\in K$, $u_k = \sum_{j=1}^n \lambda'_{k,j} x_j$ and for $l\in L$, $v_l = \sum_{j=1}^n \lambda''_{l,j} x_j$, and get 
\[
\mathsf{D} \models (\forall \overline{x})[ (\exists y)\psi(2y,\overline{x}) \Leftrightarrow \bigcurlyvee_{k\in K,l\in L} \psi(u_k + v_l,\overline{x})],  
\]
noting that also
\[
\mathsf{D} \models (\forall \overline{x})[(\exists y)\psi(y,\overline{x}) \Leftrightarrow  (\exists y)\psi(2y, \overline{x})].\qedhere
\]
\end{proof}

\begin{corollary}[{\cite[Corollary 2.2]{Pierce1972b}}]
The class of totally ordered abelian groups has the amalgamation property. 
\end{corollary}

\begin{proof}
Note that if  $\mathsf{D}$ is the class of non-trivial divisible totally ordered abelian groups, then $\iso\sub(\mathsf{D})$ is the class of totally ordered abelian groups, since every totally ordered abelian group embeds into a divisible one (just consider the divisible hull of its group reduct and order it in the obvious way). Thus, by Theorem~\ref{t:QE-AP} and Proposition~\ref{p:divisible-QE}, the class of totally ordered abelian groups has the \prp{AP}. Moreover, for the \prp{AP} it does not matter whether we consider the language with $<$ or the language with $\meet$ and $\join$.
\end{proof}

Now, it immediately follows from Corollary~\ref{c:ess-AP-comsem}  that the variety of abelian $\ell$-groups has has the \prp{AP}.

\begin{corollary}[{\cite[Theorem 2.3]{Pierce1972b}}]
$\ALG$ has the amalgamation property.
\end{corollary}

As of today, $\ALG$ is the only non-trivial subvariety of $\LG$ that is known to have the \prp{AP}. Pierce showed in \cite{Pierce1972b} that $\LG$ itself does not have the \prp{AP}, and, much later, Gurchenkov showed in \cite{Gurchenkov1997} that any variety of $\ell$-groups with the \prp{AP} is representable, i.e. that $\amal(\LG)=\amal(\RLG)$. On the other hand, Glass, Saracino, and Wood showed in \cite{Glass1984} that $\RLG$ does not have the \prp{AP}, and that the \prp{AP} also fails for some subvarieties of $\RLG$. This result was strengthen by Powell and Tsinakis in \cite{PowellTsinakis1989a}, where they show that any variety of representable $\ell$-groups that contains one of the Medvedev varieties $\cls{M}^+$ or $\cls{M}^-$ (see, e.g., \cite{KM94} for a definition) does not have the \prp{AP}. Thus, any non-abelian members of $\amal(\LG)$ must satisfy some rather particular conditions. Whether such a variety exists is a long-standing mystery.

\begin{question}
Is there any non-abelian member of $\amal(\LG)=\amal(\RLG)$?
\end{question}

The proof of Powell and Tsinakis in \cite{PowellTsinakis1989b} (see also \cite{GLT2015}) that $\RLG$ does not have the \prp{AP} relies on the fact that representable $\ell$-groups have unique roots, i.e.,
the variety of representable $\ell$-groups satisfies the following quasiequation for each $n\in \N{\setminus}\{0\}$:
\[
x^n\approx y^n \Rightarrow x\approx y.
\]
To see this, let $\m{A}$ be a totally ordered group and $a,b\in A$. If $a<b$, then $a^2 < ab < b^2$ and, inductively, $a^n < b^n$, i.e., $\m{A} \models x^n\approx y^n \Rightarrow x\approx y$. Now the claim follows, since the variety of representable $\ell$-groups is generated as a quasivariety by its totally ordered members.

To show the failure of the \prp{AP}, Powell and Tsinakis consider cyclic extensions of totally ordered groups. Let $\m{G}$ be a totally ordered group and $\alpha$ an automorphism of $\m{G}$. Then the \emph{cyclic extension of $\m{G}$ by $\alpha$} is the totally ordered group  
\[
\m{G}(\alpha) = \alg{G\times\{\alpha^n \mid n\in \Z \}, \meet, \join, \cdot, {}^{-1}, \pair{\ut, \mathsf{id}}},
\]
where 
\[
\pair{g,\alpha^m} \cdot \pair{h,\alpha^n} = \pair{g\alpha^{m}(h),\alpha^{m+n}}, \quad \pair{g,\alpha^m}^{-1} = \pair{\alpha^{-m}(g^{-1}),\alpha^{-m}},
\] 
and the order is defined by 
\[
\pair{g,\alpha^{m}} > \pair{\ut,\mathsf{id}} \iff (n> 0) \text{ or } (n= 0 \text{ and } g> \ut).
\]
If we view $\m{G}$ as a subalgebra of $\m{G}(\alpha)$ by identifying $g\in G$ with $\pair{g,id}$ and  identify  $\alpha$ with the element $\pair{\ut,\alpha}$, then we have $\alpha g \alpha^{-1} = \alpha(g)$.
Now, the idea is to find, for given $n>0$, a totally ordered group $\m{G}$ with automorphisms $\alpha, \beta, \gamma$ such that $\alpha = \beta^n = \gamma^n$, but $\beta \neq \gamma$. With such automorphisms, the span $\pair{\m{G}(\alpha) \hookrightarrow \m{G}(\beta), \m{G}(\alpha) \hookrightarrow \m{G}(\gamma)}$ does not have an amalgam in $\RLG$: Since $\beta \neq \gamma$, there exists $g\in G$ such that $\beta g \beta^{-1} = \beta(g) \neq \gamma(g) = \gamma g \gamma^{-1}$, but, since $\alpha = \beta^n = \gamma^n$ and roots are unique, $\beta$ and $\gamma$ need to be identified in an amalgam. In \cite{PowellTsinakis1989b}, Powell and Tsinakis describe a suitable totally ordered group and construct automorphisms that satisfy the desired properties. They go on to conclude the failure of \prp{AP} for various varieties of representable $\ell$-groups. In fact, it is shown in \cite{PowellTsinakis1989b} that any variety of representable $\ell$-groups that contains one of the Medvedev varieties $\cls{M}^+$ or $\cls{M}^-$  fails the \prp{AP}, which yields that uncountably many varieties of representable $\ell$-groups fail the \prp{AP}.

\begin{theorem}[\cite{Glass1984,PowellTsinakis1989b}]\label{t:APfailRLG}
$\RLG$ does not have the amalgamation property. There are uncountably many subvarieties of $\RLG$ that do not have the \prp{AP}.
\end{theorem}

The failure of the \prp{AP} for $\LG$ and $\RLG$ is used by Gil-Férez, Ledda, and Tsinakis in \cite{GLT2015} to show that various varieties of residuated lattices also fail the \prp{AP}. We will focus on the semilinear ones. Let us call a residuated lattice $\m{A}$ \emph{fully distributive} if it satisfies the equations $x \meet (y \join z) \approx (x \meet y) \join (x \meet z)$ and $x(y \meet z) w \approx xyw \meet xzw$.  Note that, in particular, every semilinear residuated lattice is fully distributive.

\begin{lemma}[{\cite[Lemma 4.1]{GLT2015}}]
If $\m{A}$ is fully distributive, then its set of invertible elements forms a subalgebra of $\m{A}$ that is an $\ell$-group. 
\end{lemma}

The following theorem is a consequence of the previous lemma.

\begin{theorem}[{\cite[Theorem 4.2]{GLT2015}}]\label{t:APfailFD}
If $\V$ is a variety of fully distributive residuated lattices such that $\V \cap \LG$ does not have the \prp{AP}, then $\V$ does not have the \prp{AP}.
\end{theorem}

Now, observe that $\RLG = \SemRL\cap\LG = \CanSRL\cap\LG$. We may thus deduce the following result immediately from Theorems~\ref{t:APfailFD} and \ref{t:APfailRLG}.

\begin{theorem}[{\cite[Theorem 4.3]{GLT2015}}]
Neither $\SemRL$ nor $\CanSRL$ has the \prp{AP}.
\end{theorem}

Notice that the proof of the preceding theorem cannot be adapted in the presence of commutativity since $\CSemRL\cap\LG = \CCanSRL\cap\LG = \ALG$, which has the \prp{AP}. This implicates the following question, which appears to demand new techniques.

\begin{question}[{\cite[Problem 6]{GLT2015}}]
Does the variety of commutative cancellative semilinear residuated lattices have the amalgamation property?
\end{question}

It also appears that these methods are insufficient to deal with the addition of knotted rules, in particular integrality. 

\begin{question}[{\cite[Problem 2]{GLT2015}}]
Does the variety of integral cancellative semilinear residuated lattices have the amalgamation property?
\end{question}

\subsection{Wajsberg Hoops and MV-algebras}\label{sec:MV}
Beyond the rather incomplete picture painted in Section~\ref{sec:lgroups}, it appears that there is little known about the \prp{AP} in cancellative varieties. However, much more is known for some prominent varieties that are constructed from the cancellative ones.

Probably the best known of these is the variety of \emph{MV-algebras}. These are bounded integral commutative semilinear residuated lattices that satisfy the identity $(x \to y) \to y  \approx x \join y$, and consequently generalize Boolean algebras. They have been studied extensively as algebraic models of the infinite-valued {\L}ukasiewicz propositional logic. The analogues of these without including a designated least element are \emph{Wajsberg hoops}; explicitly, these are integral commutative semilinear residuated lattices that satisfy $(x \to y) \to y  \approx x \join y$. Wajsberg hoops are exactly the bottom-free subreducts of MV-algebras; see \cite{Agliano2003}. We denote the variety of MV-algebras by $\MV$ and the variety of Wajsberg hoops by $\WH$.

The connection between the aforementioned algebras and cancellative residuated lattices comes from the well-known categorical equivalence of Mundici \cite{Mundici1986}. The latter yields that MV-algebras can be represented as intervals of abelian $\ell$-groups. More precisely, if $\m{G} = \alg{G,\meet,\join, +, -, 0}$ is an abelian $\ell$-group and $u\in G$, $u\geq 0$, then we can define the MV-algebra $\m{\Gamma}(\m{G},u) = \alg{[0,u],\meet,\join, \cdot, \to ,u,0}$, where meet and join are just the restrictions of the meet and join of $\m{G}$ to $[0,u]$, $a \cdot b  = (a + b -u)\join 0$, and $a \to b = (u + b - a) \meet u$. It is shown in \cite{Mundici1986} that any MV-algebra can be represented as an algebra of this form.

The following totally ordered MV-algebras are especially noteworthy:
\begin{itemize}
\item $[0,1]_{\MV} = \m{\Gamma}(\R,1)$, the standard MV-algebra;
\item $\Ln{n} = \m{\Gamma}(\Z,n)$, the $n+1$-element totally ordered MV-algebra (for $n\in \N)$;
\item $\Lnw{n} = \m{\Gamma}(\Z\overrightarrow{\times} \Z, \pair{n,0})$, where $\Z\overrightarrow{\times} \Z$ denotes the lexicographic product of the $\ell$-group of the integers with itself (for $n\in \N)$.
\end{itemize}
We denote the corresponding Wajsberg hoop reducts by $[0,1]_{\WH}$, $\Wn{n}$, and $\Wnw{n}$, respectively. Moreover, we will denote by $\m{Z}^-$ the Wajsberg hoop which is the negative cone of the $\ell$-group of integers.

In \cite{DiNola2000}, Di~Nola and Lettieri completely characterize the varieties of MV-algebras with the \prp{AP}. These are exactly the varieties generated by a single totally ordered MV-algebra.

\begin{theorem}[{\cite[Theorem 13]{DiNola2000}}]\label{t:AP-MV}
For a variety $\V$ of MV-algebras the following are equivalent:
\begin{enumerate}[\normalfont (1)]
\item $\V$ has the amalgamation property.
\item $\V$ is generated by one totally ordered MV-algebra.
\item $\V$ is of the form $\vr(\m{A})$ with $\m{A} \in \{[0,1]_\MV\} \cup \{\Ln{n},\Lnw{n} \mid n\in \N\}$.
\end{enumerate}
\end{theorem}

In \cite{Metcalfe2014}, Metcalfe, Montagna, and Tsinakis also characterize the varieties of Wajsberg hoops with the \prp{AP}.

\begin{theorem}[{\cite[Theorem 63]{Metcalfe2014}}]\label{t:AP-WH}
A variety $\V$ of Wajsberg hoops has the amalgamation property if and only if one of the following holds.
\begin{enumerate}[\normalfont (1)]
\item $\V = \vr(\m{A})$ with $\m{A} \in \{[0,1]_\WH, \m{Z}^-\} \cup \{\Wn{n},\Wnw{n} \mid n\in \N\}$.
\item $\V = \vr(\Wn{n},\m{Z}^-)$ for some $n\geq 1$.
\end{enumerate}
\end{theorem}

Alternative proofs of these results may also be given with relative ease using the techniques discussed in Section~\ref{sec:tools}, as the reader may verify.

\subsection{Basic Hoops and BL-algebras}\label{sec:BL}
The characterizations of $\amal(\MV)$ and $\amal(\WH)$ are rather clean, and it is natural to wonder whether they can be pushed to more complicated subvarieties of $\CSemRL$ for which MV-algebras or Wajsberg hoops play a significant structural role. Probably the most significant varieties of this kind are the variety of \emph{BL-algebras} and the variety of \emph{basic hoops}.
Formally a \emph{BL-algebra} is a bounded integral commutative semilinear residuated lattice that satisfies the \emph{divisibility} condition, i.e., the formula
\[
x\leq y \implies \exists z (yz \approx x).
\]
Divisibility can be captured equationally by the identity $x(x \to y) \approx x \meet y$, so BL-algebras form a variety denoted by $\BL$. Notably, $\BL$ contains both $\MV$ and $\GA$ as subvarieties. BL-algebras have been highly influential as models of fuzzy logics, and indeed give the equivalent algebraic semantics of Petr H\'{a}jek's basic fuzzy logic, which is the logic of continuous triangular norms; see \cite{Hajek1998,CEGT2000}. \emph{Basic hoops} dispense with the requirement of being bounded, i.e., they are integral commutative semilinear residuated lattice that satisfies the equation $x(x \to y) \approx x \meet y$. The variety of basic hoops is denoted $\BH$. 

BL-algebras and basic hoops are very closely linked to MV-algebras and Wajsberg hoops: Aglian\`{o} and Montagna showed in \cite{Agliano2003} that every totally ordered BL-algebra or basic hoop is the nested sum\footnote{The nested sum is more commonly called the \emph{ordinal sum} in the literature on BL-algebras} of a family of totally ordered Wajsberg hoops (in the BL case the last algebra in the nested sum is an MV-algebra):

\begin{proposition}[\cite{Agliano2003}]\label{prop:ordinal sum decomp}
\hfill
\begin{enumerate}[\normalfont (i)]
\item For every totally ordered basic hoop $\m{A}$ there exists a unique (up to isomorphism) totally ordered set $\alg{I,\leq}$ and unique (up to isomorphism) totally ordered Wajsberg hoop $\m{A}_i$, $i\in I$, such that $\m{A} \cong \Nsum_{i\in I} \m{A}_i$.

\item For every totally ordered BL-algebra $\m{A}$ there exists a unique (up to isomorphism) totally ordered set $\alg{I,\leq}$ with minimum $i_0$, a unique (up to isomorphism) totally ordered MV-algebra $\m{A}_{i_0}$, and  unique (up to isomorphism) totally ordered Wajsberg hoops $\m{A}_i$, $i\in I{\setminus}\{i_0\}$, such that $\m{A} \cong \Nsum_{i\in I} \m{A}_i$.

\item For every totally ordered BL-algebra $\m{A}$ there exists a  a unique (up to isomorphism) totally ordered MV-algebra $\m{A}_1$ and a unique (up to isomorphism) totally ordered basic hoop $\m{A}_2$ such that $\m{A} \cong \m{A}_1 \nsum \m{A}_2$.
\end{enumerate}
\end{proposition}

Given this beautiful structure theorem for BL-algebras and basic hoops in terms of MV-algebras and Wajsberg hoops, the idea of pushing the transparent description of $\amal(\MV)$ and $\amal(\WH)$ to BL-algebras and basic hoops is quite enticing. However, this turns out to be dramatically more difficult than it may at first seem.

The first effort to study the \prp{AP} for varieties of BL-algebras was Montagna's 2006 paper \cite{Montagna2006}, which shows that various known varieties of BL-algebras have the \prp{AP} and that there are continuum-many varieties of BL-algebras that fail the \prp{AP}. Later on, Cortonesi, Marchioni, and Montagna studied the \prp{AP} in BL-algebras from an different angle in \cite{CMM11}, where they deploy tools from first-order model theory and quantifier elimination. Further progress was made by Aguzzoli and Bianchi in \cite{AB21}, where a partial classification of varieties of BL-algebras with the \prp{AP} is given. They extended this partial classification in \cite{AB23}, after the characterization of \prp{AP} in Theorem~\ref{t:APmain} had become available in an early version of \cite{Fussner2024}. A complete description of $\amal(\BL)$ was finally given in \cite{FS24}, heavily using the tools discusses in Section~\ref{sec:tools}.

In fact, \cite{FS24} proceeds by first giving an explicit description of $\amal(\BH)$, which is somewhat more transparent, and then extending this to give an explicit description of $\amal(\BL)$. Both $\amal(\BH)$ and $\amal(\BL)$ turn out to be unions of countably many finite intervals, which may be described precisely.

The characterizations of $\amal(\BL)$ and $\amal(\BH)$ proceeds as follows. First, we say that a totally ordered basic hoop or BL-algebra is \emph{of finite index} if the chain alluded to in Proposition~\ref{prop:ordinal sum decomp}(i) (respectively (ii)) is finite. For a class $\V$ of basic hoops or BL-algebras, $\Vfc$ denotes its totally ordered members of finite index. Since every finitely generated totally ordered BL-algebra or basic hoop is of finite index, it follows from Corollary~\ref{c:ess-AP-comsem} that if $\V$ is a variety, then it has the \prp{AP} if and only if $\Vfc$ has the essential \prp{AP}.

For a variety of of basic hoops $\V$, denote by $\wajs(\V)$ the class of totally ordered Wajsberg hoops in $\V$. Then $\wajs(\V) \subseteq \Vfc$ and any member of $\Vfc$ is a finite nested sum of members of $\wajs(\V)$. Moreover, for a variety $\V$ of BL-algebras, denote by $\luk(\V)$ the class of totally ordered MV-algebras in $\V$ and by $\basic(\V)$ the class of totally ordered basic hoops $\m{A}_2$ such that $\m{A}_1 \nsum \m{A}_2  \in \chain{\V}$ for some $\m{A}_1 \in \luk(\V)$. Further defining $\basicfc(\V) = \basic(\V)_{\mathrm{fc}}$, we have $\Vfc \subseteq \luk(\V) \nsum \basicfc(\V)$. 

In \cite{FS24}, it is shown that checking the \prp{AP} for a variety of basic hoops reduces to checking the \prp{AP} of its Wajsberg components together with the arrangement of said components in nested sum decompositions. Further, the \prp{AP} for a variety of BL-algebras reduces to the \prp{AP} for its MV-components and basic hoops components. In particular, by using Theorems~\ref{t:AP-MV} and~\ref{t:AP-WH}, the following result is proven:

\begin{proposition}[\cite{FS24}]\hfill
\begin{enumerate}[\normalfont (i)]
\item If a variety $\V$ of basic hoops has the amalgamation property, then one of the following holds.
\begin{itemize}
\item $\wajs(\V) = \hspu(\m{A})$ with $\m{A} \in \{[0,1]_\WH, \m{Z}^-\} \cup \{\Wn{n},\Wnw{n} \mid n\in \N\}$;
\item $\wajs(\V) = \hspu(\Wn{n},\m{Z}^-)$ for some $n\geq 1$.
\end{itemize}    
\item If a variety $\V$ of BL-algebras has the amalgamation property, then $\luk(\V) = \hspu(\m{A})$ with $\m{A} \in \{[0,1]_\MV\} \cup \{\Ln{n},\Lnw{n} \mid n\in \N\}$ and $\basicfc(\V)$ has the essential amalgamation property.
\end{enumerate}
\end{proposition}

However, the previous proposition only gives a necessary condition for the \prp{AP}. For a sufficient condition, there are some more involved technicalities regarding the nested sum decompositions of members of $\Vfc$ that have to be considered. These technicalities rely heavily on the notion of the essential \prp{AP} and essential closedness (see Section~\ref{sec:tools}). As a consequence of this analysis, one may obtain the following result:

\begin{theorem}[\cite{FS24}]\hfill
\begin{enumerate}[\normalfont (i)]
\item There are only countably many varieties of basic hoops that have the amalgamation property.
\item There are only countably many varieties of BL-algebras that have the amalgamation property.
\end{enumerate}
\end{theorem}  

In fact, \cite{FS24} gives an explicit characterization of the varieties of basic hoops and BL-algebras with the \prp{AP}, although the list of these varieties is rather complicated.

The form of the arguments in \cite{FS24} suggest that a similar analysis is likely to prevail in describing $\amal(\V)$ for other varieties $\V\in\slat(\CSemRL)$ adjacent to $\BL$ (or $\BH)$, provided that a reasonably tractable nested sum representation is available for the totally ordered members of $\V$. The most natural next step is to consider the variety $\MTL$ of MTL-algebras, i.e., bounded integral commutative semilinear residuated lattices which, as seen in Section~\ref{sec:knotted}, fails the \prp{AP}. However, nested sum representations for MTL-algebras appear to be far less useful than they are for BL-algebras, owing to the fact that the sum-irreducible MTL-algebras are extraordinarily diverse and complex. The following question is thus intriguing, but likely quite difficult.

\begin{question}
Can one explicitly describe $\amal(\MTL)$? In particular, is $\amal(\MTL)$ countable?
\end{question}

Although partial answers to the preceding question are likely to be quite easy to come by---e.g., in the context of subvarieties of $\MTL$ with transparent sum-irreducible members---it appears that a full resolution of this problem will require some novel ideas.

\appendix
\section{Nomenclature for subvarieties}\label{app:nomenclature}
In this appendix, we summarize the notational conventions for varieties used throughout this paper. We also provide a brief digest of what is known about the \prp{AP} in varieties that we have considered. All of this information is given in the following table. The last column contains what is known about the cardinality of $\amal(\V)$ for the respective varieties.

\begin{table}[t]
\centering
\begin{tabular}{|l|l|l|l|}
\hline
Variety $\V$ & Abbreviation & AP &  $\lvert \amal(\V) \rvert$ \\
\hline
\hline 
Semilinear residuated lattices & $\SemRL$ & no* & $2^{\aleph_0}$  \\
\hline
Commutative $\SemRL$ & $\CSemRL$ & no* & $\geq \aleph_0$  \\
\hline
Idempotent $\SemRL$ & $\nSRL{1}$ & no   & $2^{\aleph_0}$  \\
\hline 
Idempotent $\CSemRL$ & $\nCSRL{1}$ & yes & $60$ \\
\hline 
Gödel algebras & $\GA$ & yes  & $4$ \\
\hline 
Relative Stone algebras & $\RSA$ & yes & $3$  \\
\hline
Sugihara monoids & $\SM$ & yes  & $9$ \\
\hline 
Odd Sugihara monoids & $\OSM$ & yes & $3$ \\
\hline 
$\pair{m,n}$-knotted $\SemRL$ ($m\geq1$, $n\geq 0$) & & no* & $2^{\aleph_0}$    \\
\hline 
$\pair{m,n}$-knotted $\CSemRL$ ($m\geq1$, $n\geq 0$) & & no*  & ? \\
\hline 
$n$-potent $\SemRL$ ($n\geq 2$) &  & no*   & $2^{\aleph_0}$  \\
\hline
$n$-potent $\CSemRL$ ($n\geq 2$) &  & no* & $\geq 60$ \\
\hline
MTL-algebras & $\MTL$ & no &  $\geq \aleph_0$ \\
\hline 
De Morgan monoids & $\mathsf{DMM}$ & no & $\geq \aleph_0$ \\
\hline 
Semilinear $\mathsf{DMM}$ & $\mathsf{SDMM}$ & no & $\geq \aleph_0$ \\
\hline
Cancellative $\SemRL$ & $\CanSRL$ & no & $\geq 3$ \\
\hline 
Commutative $\CanSRL$ & $\CCanSRL$ & ? & $\geq 3$ \\
\hline 
Integral  $\CanSRL$ &  & ? & $\geq 2$ \\
\hline
Lattice-ordered groups & $\LG$ & no & $\geq 2$  \\
\hline 
Abelian $\LG$ & $\ALG$ & no & $2$  \\
\hline 
Representable $\LG$ & $\RLG$ & no & $\geq 2$ \\
\hline  
MV-algebras & $\MV$ & yes & $\aleph_0$  \\
\hline 
Wajsberg hoops & $\WH$ & yes & $\aleph_0$ \\
\hline
BL-algebras & $\BL$ & yes & $\aleph_0$  \\
\hline 
Basic hoops & $\BH$ & yes  & $\aleph_0$ \\
\hline
\end{tabular}
\caption{Table of the some of the varieties considered in this paper. The superscript * indicates that the same is true for bounded and involutive expansions and a question mark indicates that it is an open problem.}
\end{table}

%%%%%%  THE BIBLIOGRAPHY
%%%%%%  See the examples and format yours according to them
\clearpage
\bibliographystyle{plain}
\bibliography{literatur}

%%%%%%  AUTHOR'S ADDRESS INFORMATION AT THE END OF THE PAPER:

\AuthorAdressEmail{Wesley Fussner}{Institute of Computer Science\\
Czech Academy of Sciences \\
Pod Vodárenskou věží 271/2 \\
182 00 Praha 8 \\
Czechia}{fussner@cs.cas.cz}

\AuthorAdressEmail{Simon Santschi}{Mathematical Institute\\
University of Bern\\
Alpeneggstrasse 22 \\
3012 Bern\\
Switzerland}{simon.santschi@unibe.ch}

\end{document}